\documentclass[english,11pt]{article}%
\usepackage[T1]{fontenc}
\usepackage[utf8]{inputenc}
\usepackage{amsmath}
\usepackage{amssymb}
\usepackage{amsfonts}
\usepackage{babel}
\usepackage{fancyhdr}
\usepackage{color}
\usepackage[toc]{appendix}
\usepackage{graphicx}
\usepackage{geometry}%
\usepackage{algorithm}
\usepackage{algpseudocode}
\usepackage{subcaption}
\usepackage{hyperref}
\newcommand\keywordsname{Key words}
\newcommand\AMSname{AMS subject classifications}
\newenvironment{@abssec}[1]{%
     \if@twocolumn
       \section*{#1}%
     \else
       \vspace{.05in}\footnotesize
       \parindent .0in
         {\upshape\bfseries #1. }\ignorespaces
     \fi}
     {\if@twocolumn\else\par\vspace{.1in}\fi}
\newenvironment{keywords}{\begin{@abssec}{\keywordsname}}{\end{@abssec}}

\newenvironment{AMS}{\begin{@abssec}{\AMSname}}{\end{@abssec}}

\newcommand{\tr}{\text{tr}}
\newcommand{\R}{\mathbb{R}}
\providecommand{\U}[1]{\protect\rule{.1in}{.1in}}
\makeatletter
\newtheorem{theorem}{Theorem}

\newtheorem{condition}[theorem]{Condition}

\newtheorem{definition}[theorem]{Definition}
\newtheorem{example}{Example}

\newtheorem{lemma}[theorem]{Lemma}

\newtheorem{remark}[theorem]{Remark}

\numberwithin{equation}{section}
\numberwithin{example}{section}
\numberwithin{theorem}{section}

\newenvironment{proof}[1][Proof]{\noindent\textbf{#1.} }{\ \rule{0.5em}{0.5em}}
\makeatother
\begin{document}

\title{Particle exchange Monte Carlo methods for eigenfunction and related nonlinear problems\thanks{Division of Applied Mathematics, Brown University. Research supported
in part by the the Air Force Office of Scientific Research
(FA9550-21-1-0354).}}
\author{Paul Dupuis and Benjamin J. Zhang}
\maketitle

\begin{abstract}
We introduce and develop a novel particle exchange Monte Carlo method.
Whereas existing methods apply to eigenfunction problems where the eigenvalue is known
(e.g., integrals with respect to a Gibbs measure, which can be interpreted 
as corresponding to eigenvalue zero),
here the focus is on problems where the eigenvalue is not known \emph{a priori}.
To obtain an appropriate particle exchange rule we must consider a pair of processes,
with one evolving forward in time and the other backward.
Applications to eigenfunction problems corresponding to quasistationary distributions and ergodic stochastic control are discussed.
\end{abstract}



\section{Introduction}

Markov chain Monte Carlo is widely used for approximations in linear problems,
such as computing integrals with respect to a Gibbs measure. The method can be
effective even when a straightforward implementation suffers from the ``rare
event sampling problem,'' which means that different parts of the state space
over which integration takes place do not communicate well under the Markov
chain whose stationary distribution is the target measure. 
Here one can
make use of interacting particle methods such as parallel tempering to
mitigate the communication problem, by linking the chain responsible for
sampling from the target distribution with chains that sample from related
distributions which suffer less from rare event issues \cite{gey,swewan}.

A simple example is the problem where the target measure is
\[
\mu^{\varepsilon}(dx)=\frac{1}{Z_{\varepsilon}}e^{-\frac{2}{\varepsilon}%
V(x)}dx.
\]
Here $Z_{\varepsilon}$ is the normalization constant that makes the measure a
probability, $V:\mathbb{R}^{d}\rightarrow\mathbb{R}$ is some smooth potential
that tends to $\infty$ as $\left\Vert x\right\Vert \rightarrow\infty$, so as to
make $Z_{\varepsilon}$ finite, and $\varepsilon$ is a parameter proportional
to temperature in physical models. Let $DV(x)$ denote the gradient of $V$. One
easily checks that the diffusion
\begin{equation}
dX^{\varepsilon}=-DV(X^{\varepsilon})dt+\sqrt{\varepsilon}dW,\label{eqn:SDE}%
\end{equation}
where $W$ is a standard Brownian motion on $\R^d$, 
is reversible with respect to $\mu^{\varepsilon}$, and hence the empirical
measure $\frac{1}{T}\int_{0}^{T}\delta_{X_{t}^{\varepsilon}}(dx)dt$ converges
by the ergodic theorem to $\mu^{\varepsilon}(dx)$ as $T\rightarrow\infty$.
However, when $\varepsilon>0$ is small and $V$ has multiple local minima, it
takes $X_{t}^{\varepsilon}$ a long time to move between the local minima, and
so convergence is slow.

In the most common formulation of particle exchange Monte Carlo methods, temperature is used to identify the
dynamics of the other particle system. For example, if there are only two particles, one would
take $\tau>\varepsilon$ and dynamics $dX^{\tau}=-DV(X^{\tau})dt+\sqrt{2\tau
}d\bar{W}$, with the corresponding stationary distribution $\mu^{\tau}$, and
Brownian motion $\bar{W}$ independent of $W$. The distribution $\mu^{\tau}$ is
\textquotedblleft flatter,\textquotedblright\ and the higher diffusivity of
$X^{\tau}$ allows it to wander around the state space more easily. The
stationary distribution of the pair $(X^{\varepsilon},X^{\tau})$ is of course
the product measure. Then additional randomness is introduced by allowing the
particles to swap locations at random times. 
A Metropolis--type swap mechanism between the higher and lower temperature dynamics that preserves the invariant distribution is used. One intuitively expects that the higher temperature dynamics, linked to those of the lower temperature in this way, to speed up the convergence of the low temperature marginal of the empirical measure. This expectation is true, and rigorous statements regarding the
improvement can be made \cite{dupliupladol,dupwu2}.

We previously remarked that the problems to which such parallel tempering methods apply are \emph{linear}, and that the methods exploit a reversibility structure. By linear, we mean that the algorithms compute approximations to the solution to the eigenfunction problem
\begin{equation}
-\mathcal{L}^{\varepsilon,\ast}\rho^{\varepsilon}(x)=0\text{ for }%
x\in\mathbb{R}^{d}, \label{eqn:SD}%
\end{equation}
where $\mathcal{L}^{\varepsilon,\ast}$ is the adjoint of the generator of
$X^{\varepsilon}$, and $\rho^{\varepsilon}(x)$ is normalized to be a
probability density. Although an eigenfunction problem, it is linear since the
eigenvalue is known.

Among several goals of the the present paper, perhaps the most important is to
introduce a computational approach to a class of nonlinear problems that in
some sense generalizes parallel tempering. Although the class of nonlinear
problems that can be so treated is more general (as will be discussed later),
to connect most directly to the previous discussion we use as an example an
eigenfunction problem of the form%
\begin{equation}
-\mathcal{L}^{\varepsilon,\ast}\psi^{\varepsilon}(x)+c(x)\psi^{\varepsilon
}(x)=\lambda^{\varepsilon}\psi^{\varepsilon}(x)\text{ for }x\in\mathbb{R}^{d},
\label{eqn:QSD}%
\end{equation}
where $c:\mathbb{R}^{d}\rightarrow\lbrack0,\infty)$. Under suitable conditions, there will be a real-valued eigenvalue with the largest real part
among all eigenvalues, and the eigenfunction of interest is positive, unique
up to a multiplicative constant, and corresponds to this largest eigenvalue.
Our interest is to approximate $\lambda^{\varepsilon}$ and $\psi^{\varepsilon
}$ through Monte Carlo methods. 

The most straightforward stochastic interpretation for $(\psi^{\varepsilon
}(\cdot),\lambda^{\varepsilon})$ that might serve as a basis for Monte Carlo
is the following. Under appropriate conditions one has, for any initial
condition $X_{0}^{\varepsilon}$,
\begin{equation}
\psi^{\varepsilon}(x)dx=\lim_{T\rightarrow\infty}\frac{Ee^{-\int_{0}%
^{T}c(X_{t}^{\varepsilon})dt}1_{\{dx\}}(X_{T}^{\varepsilon})}{Ee^{-\int
_{0}^{T}c(X_{t}^{\varepsilon})dt}} \label{eqn:largeT}%
\end{equation}
and
\[
\lambda^{\varepsilon}=-\lim_{T\rightarrow\infty}\frac{1}{T}\log Ee^{-\int
_{0}^{T}c(X_{t}^{\varepsilon})dt}.
\]
One can interpret the discounting term $e^{-c(X_{t}^{\varepsilon})}$ as
killing of the process at rate $c(X_{t}^{\varepsilon})$, and thus
$\psi^{\varepsilon}(x)dx$ is the asymptotic (as $T\rightarrow\infty$) limit of
the distribution of $X_{T}^{\varepsilon}$, conditioned on not having been
killed by time $T$. Also, the eigenvalue is the asymptotic killing rate,
conditioned on survival. Although this is perhaps the simplest stochastic
representation, we will use as the starting point a representation in terms of
a closely related process that does not die out. In probability,
$\psi^{\varepsilon}(x)dx$ is called the \emph{quasi-stationary} distribution for
these dynamics with killing rate $c(x)$.\footnote{More precisely, this is
\textquotedblleft soft\textquotedblright\ killing, to be contrasted with the
case of \textquotedblleft hard\textquotedblright\ killing, in which the
process is killed the first time it exits a given set.}

Stochastic methods used to solve eigenfunction problems of this sort generally
suffer from the same difficulties as those used for stationary distributions
when there is a rare event sampling problem. Hence one may ask if there is
any analogue to parallel tempering for Monte Carlo schemes that might be used
for their approximation. It turns out that this is possible, and the key
insight needed to make this work is that one cannot use just the process
$X_{t}^{\varepsilon}~$and qualitatively similar processes, such as the higher
temperature variants used in parallel tempering's most common implementation.
Rather, there is a \textquotedblleft dual\textquotedblright\ dynamic based on
time reversal that must be first linked to the original dynamic, and one can
then work with the forward/backward pair to obtain a reversibility structure
and a consistent Metropolis type swapping rule. In fact, once one understands
how to swap the forward/backward pair one can then couple to other
forward/backward pairs that are related to the first one, for example by a
change in a temperature parameter. The main task of the paper is to elucidate
the structure of the forward/backward pair, and show how it is used in
developing the analogue of parallel tempering for nonlinear problems.
To illustrate the ideas we focus on a diffusion model.
However, the approach generalizes without difficulty,
and in particular can also be used for countable state and finite state jump Markov processes.

We conclude this introduction by mentioning various issues and objects that
will be needed for or which will appear naturally in the analysis. The role
that each plays is a topic that will need further investigation. However, as
noted the purpose of the present paper is to introduce and link the needed
ideas in just one context, and then outline the use of the methods for some
problems of interest. We should also 
make precise what it means to \textquotedblleft
solve\textquotedblright\ a problem, which can depend on context. One meaning will be the
accurate approximation of a given integral by Monte Carlo. Another meaning is
that samples generated by a sampling scheme have been combined with a machine learning method
(e.g., kernel or neural
net) to create an accurate approximation to an eigenfunction, or in 
control problems to generate an
approximation to the derivative of a cost potential, and hence an approximation
to the optimal control.

\begin{itemize}
\item When using parallel tempering to solve \eqref{eqn:SD} the starting point
is an ergodic Markov process with $\rho^{\varepsilon}(x)$ as its stationary
distribution. The starting point for (\ref{eqn:QSD}) is analogous, but this
time in terms of what is called a ``nonlinear Markov process'' (NLMP), and so
in Section \ref{sec:nonlinearMPswap} we introduce such processes and discuss a
second stochastic representation for \eqref{eqn:QSD}, by which the
eigenfunction is characterized as the stationary distribution of the nonlinear process.

\item For the NLMP, there is not a direct analogue of the structure that
allows linking, e.g., similar processes that differ only through different
temperature parameters. This can only be done after, as a first step, one has
linked the process with a dual process, which also has an interpretation as a
NLMP (though in special cases it might be an ordinary Markov process). This
connection is established in Section \ref{sec:swap}, where it is shown that a
Metropolis type rule can be constructed for the pair of processes that allows
them to swap locations without perturbing their respective stationary
distributions. It is important to note that this imposes restrictions on the
process models (or in other words the operators appearing in the eigenfunction
problems), analogous to reversibility of processes used in approximating Gibbs
measures through a Metropolis-type construction.

\item One can argue that according to various measures of convergence, the
Monte Carlo approximations generated by the pair are optimized when passing to
the infinite swapping (INS) limit. Section \ref{sec:infswap} derives this
limit, as well as an appropriate weighted empirical measure that replaces the
ordinary empirical measure used at the prelimit. Also included in this section
is a discussion on how the problem of sampling from a Gibbs distribution is
included as a special case.
Section \ref{sec:imppot} introduces the
\textquotedblleft implied potential\textquotedblright\ that describes the
energy landscape of the stationary distribution that is relevant for the INS
algorithm and contrasts it with the landscape of the original dynamics.

\item While the construction up to this point is correct theoretically, it
is not practical, in that NLMPs cannot be simulated directly. What is arguably
the most natural approximation to the NLMPs is through a Fleming-Viot (FV)
interacting particle model. Once this approximation has been made direct
simulation is possible. The connection between FV systems and problems from
analysis seems to have originated with \cite{burholingmar}, and their use as a
computational tool seems to have been first suggested and developed by D.
Villemonais (see, for example, \cite{vil2}). With the introduction of the FV
system there are now two types of interaction--that within a FV process, and
that which produces swaps between different particles in the two FV processes,
one approximating the NLMP forward dynamics and a second approximating the
different NLMP backward dynamics.

\item Since we consider in this paper diffusion processes, there is also a
need to discretize such processes in time. For reasons explained in Section
\ref{sec:pure_jump}, we opt for a pure jump Markov chain approximation rather
than an Euler-Maruyama type scheme.
\end{itemize}

An outline is as follows. Section 2 defines the pair of NLMPs needed for a
legitimate swapping rule, and also formulates the INS limit.
Section 2.4 provides an illustration of the improved communication properties due to the swapping.
Section 3 presents the approximation of the NLMPs via Fleming-Viot processes,
while Section 4 presents the discrete time approximation to the continuous time trajectories.
One of the motivating    applications is to ergodic stochastic control,
and Section 5 describes how samples from the distribution defined by an eigenfunction can be used to directly approximate the cost potential and hence the optimal control.
Section 6 presents numerical examples, and open problems and concluding remarks are given in Section 7. We summarize the main ideas of this paper and the connections between them in Figure~\ref{fig:conceptdiagram}. We encourage the reader to use the concept diagram as a guide when reading this paper. 
\begin{figure}
    \centering
    \includegraphics[width=0.88\linewidth]{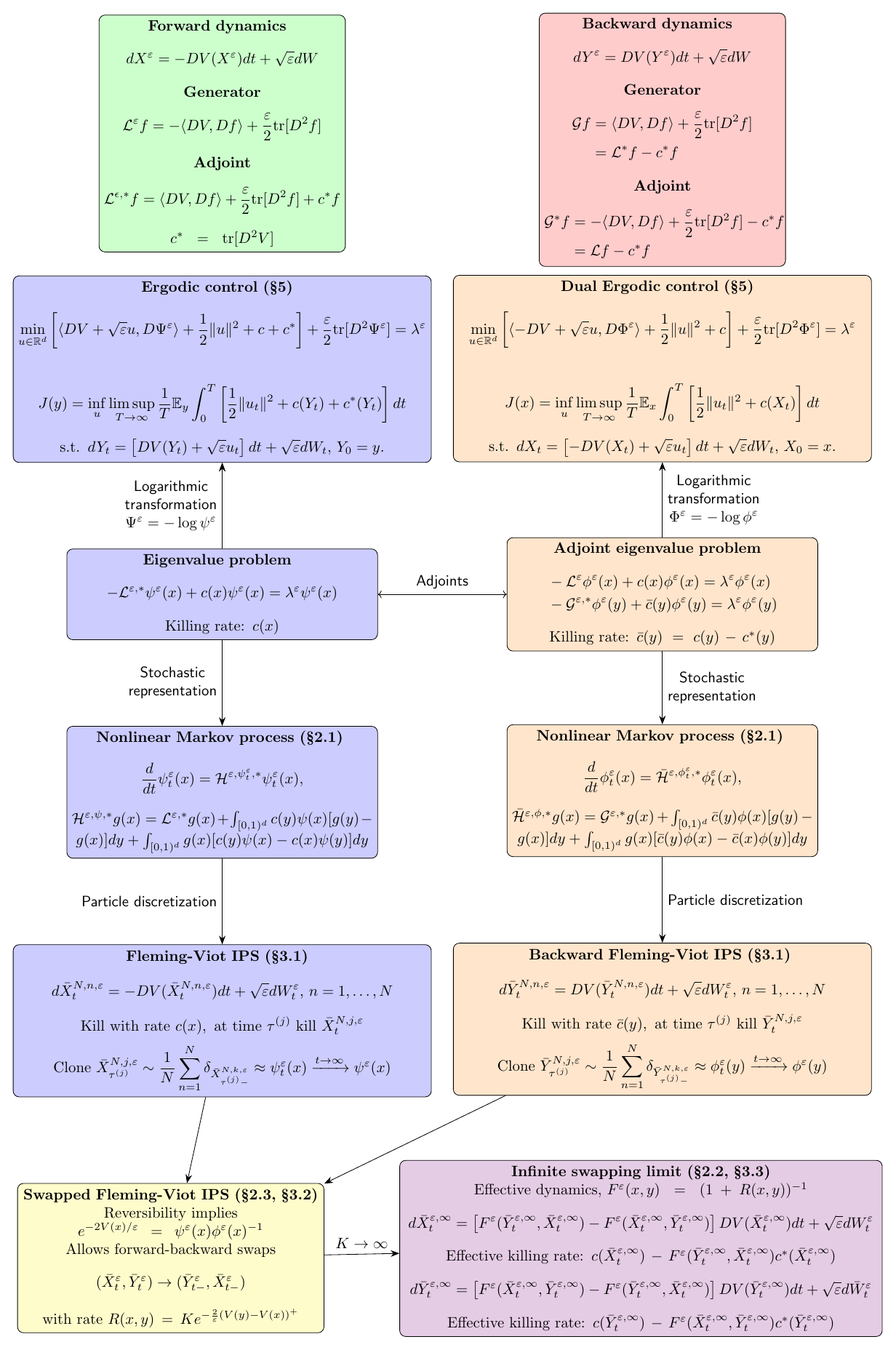}
    \caption{Concept diagram of the major ideas in this paper. }
    \label{fig:conceptdiagram}
\end{figure}

\section{Swapping between a forward--backward pair of nonlinear Markov
processes}

\label{sec:2}

\label{sec:nonlinearMPswap} In this section we formulate a
swapping mechanism to help with the Monte Carlo approximation of
(\ref{eqn:QSD}). As with (\ref{eqn:SD}), the approximation is through a large
time average. Unlike with (\ref{eqn:SD}), for which one can use a Markov
process model, for an equation of the form (\ref{eqn:QSD}) the starting point
is a \emph{nonlinear} Markov process (NLMP). It is not feasible to simulate the nonlinear Markov process exactly, but in Section~\ref{sec:fleming-viot}, we introduce the Fleming-Viot interacting particle system which provides an implementable approximation of the NLMP. 

The diffusion (\ref{eqn:SDE}) has the generator
\[
\mathcal{L}^{\varepsilon}f(x)=-\left\langle DV(x),Df(x)\right\rangle
+\frac{\varepsilon}{2}\text{tr}[D^{2}f(x)]
\]
for $f\in C_{c}^{2}(\mathbb{R}^{d})$, the space of twice continuously
differentiable functions from $\mathbb{R}^{d}$ to $\mathbb{R}$ with compact
support. Here $D^{2}f(x)$ denotes the Hessian. To avoid potential
nonuniqueness issues in the noncompact setting, we assume that $V$ is periodic
with $V(x)=V(x+e_{i})$ for any $x\in\mathbb{R}^{d}$ and any $i\in
\{1,\ldots,d\}$, where $e_{i}$ is the $i$th unit vector in the standard basis.
The formal adjoint is given by
\[
\mathcal{L}^{\varepsilon,\ast}f(x)=\left\langle DV(x),Df(x)\right\rangle
+\frac{\varepsilon}{2}\text{tr}[D^{2}f(x)]+c^{\ast}(x)f(x),
\]
where $c^{\ast}(x)=%
{\textstyle\sum\nolimits_{i=1}^{d}}
\partial^{2}V(x)/\partial^{2}x_{i}=\mbox{tr}[D^{2}V(x)]$. Under the
periodicity condition, it is known that there is a positive eigenvalue
$\lambda^{\varepsilon}$ with the largest real part among all eigenvalues, and
corresponding positive eigenfunction $\psi^{\varepsilon}$ that is unique up to
a multiplicative constant, that together solve (\ref{eqn:QSD}). When
normalized so that the integral over $[0,1]^{d}$ equals $1$, $\psi
^{\varepsilon}$ is the QSD mentioned in the Introduction, and $\lambda
^{\varepsilon}$ the asymptotic killing rate, conditioned on survival. To
simplify the discussion we assume that both $c(x)$ and $\bar{c}%
(x)=c(x)-c^{\ast}(x)$ are non-negative and continuous, but stress that these
conditions are not needed but used so we can consider for now stochastic
representations that involve only killing, rather than killing and cloning.

While we suppose that our primary interest is in the solution to
(\ref{eqn:QSD}), a second problem that will be relevant is
\begin{equation}
-\mathcal{L}^{\varepsilon}\phi^{\varepsilon}(x)+c(x)\phi^{\varepsilon
}(x)=\lambda^{\varepsilon}\phi^{\varepsilon}(x)\text{ for }x\in\mathbb{R}^{d}.
\label{eqn:QSDdual}%
\end{equation}
This problem clearly has the same structure as (\ref{eqn:QSD}), and its
solution identifies the QSD for the process
\begin{equation}
dY^{\varepsilon}=DV(Y^{\varepsilon})dt+\sqrt{\varepsilon}d\bar{W}
\label{eqn:SDEdual}%
\end{equation}
and killing rate $\bar{c}(x)=c(x)-c^{\ast}(x)$.\footnote{Another probabilistic
interpretation for $\phi^{\varepsilon}(x)$ is that $\phi^{\varepsilon}%
(x_{1})/\phi^{\varepsilon}(x_{1})$ gives a large $t$ approximation to ratio of
the probabilities of survival past $t$ for different starting points.} The
eigenvalue with largest real part is again $\lambda^{\varepsilon}$, and as a
QSD on $[0,1]^{d}$, $\phi^{\varepsilon}(x)$ is strictly positive. Note that
the dynamics of $Y^{\varepsilon}$ are the time reversal of the dynamics of
$X^{\varepsilon}$.

\begin{condition}
\label{con:reg}Each component of $DV(x)$ is $C^{1,\alpha}$ for some $\alpha
\in(0,1]$, and $c(x)$ is $C^{\alpha}$ for the same $\alpha$.
\end{condition}

We next state some standard results concerning the eigenfunctions
$\phi^{\varepsilon}$ and $\psi^{\varepsilon}$ under Condition \ref{con:reg}.
Let $\hat{p}_{t}^{\varepsilon}(x,y)$ denote the Green's function for
$\frac{\partial u}{\partial t}=\mathcal{L}^{\varepsilon}u-cu$, so that
\[
E_{x}e^{-\int_{0}^{t}c(X_{s}^{\varepsilon})ds}f(X_{t}^{\varepsilon}%
)=\int_{[0,1)^{d}}\hat{p}_{t}^{\varepsilon}(x,y)f(y)dy,
\]
and note that $\hat{p}_{t}^{\varepsilon}(x,y)dy$ is a sub-probability.

\begin{theorem}
Assume Condition \ref{con:reg}. Then there is $\lambda^{\varepsilon}%
\in(0,\infty)$ which is a simple eigenvalue of $-$ $[\mathcal{L}%
^{\varepsilon,\ast}-c]$, such that $\operatorname{Re}(\gamma)>$ $\lambda
^{\varepsilon}$ for any other eigenvalue $\gamma$, the associated
eigenfunction $\psi^{\varepsilon}$ is positive, there is no other eigenvalue
with a positive eigenfunction, and $\psi^{\varepsilon}$ is $C^{2,\alpha}$ for
some $\alpha\in(0,1]$. The analogous statement holds for $-$ $[\mathcal{L}%
^{\varepsilon}-\bar{c}]$ with the same eigenvalue and eigenfunction
$\phi^{\varepsilon}$. Lastly, there is $\delta>0$ such that if $\phi
^{\varepsilon}$ and $\psi^{\varepsilon}$ are normalized so that $\int
_{[0,1)^{d}}\psi^{\varepsilon}(x)dx=1$ and $\int_{[0,1)^{d}}\phi^{\varepsilon
}(x)\psi^{\varepsilon}(x)dx=1$, then for $t\geq1$
\[
\hat{p}_{t}^{\varepsilon}(x,y)=e^{-\lambda^{\varepsilon}t}\phi^{\varepsilon
}(x)\psi^{\varepsilon}(y)\left[  1+O(e^{-\delta t})\right]
\]
with $\delta$ independent of $x$ and $y$.
\end{theorem}

\begin{proof}
The claims regarding the existence of a principal eigenvalue and positive
eigenfunction can be found in \cite[Chapter 1]{du} and \cite[Chapter 3,
Theorem 10.1]{har7}. (Du does not consider the case of periodic boundary
condition, but the argument used for the Dirichlet, Neumann and Robin boundary
conditions applies with little change to that case.) The regularity is from
\cite[Chapter 2]{du}, while the expansion of the Green's function is also in
\cite[Chapter 3, Theorem 10.1]{har7}.
\end{proof}

\medskip As a consequence of the theorem, we see that given any starting
distribution $\nu^{\varepsilon}\in\mathcal{P}([0,1)^{d})$,
\[
\frac{E_{\nu^{\varepsilon}}e^{-\int_{0}^{t}c(X_{s}^{\varepsilon})ds}%
1_{\{dy\}}(X_{t}^{\varepsilon})}{E_{\nu^{\varepsilon}}e^{-\int_{0}^{t}%
c(X_{s}^{\varepsilon})ds}} =\frac{\int_{[0,1)^{d}}\hat{p}_{t}^{\varepsilon
}(x,y)dy\nu^{\varepsilon}(dx)}{{\int_{[0,1)^{d}}\int_{[0,1)^{d}}\hat{p}%
_{t}^{\varepsilon}(x,y)dy\nu^{\varepsilon}(dx)}} \rightarrow\psi^{\varepsilon
}(y)dy.
\]

\medskip We next give the formal definition of a QSD (with soft killing) and
make precise the correspondence between solutions of the eigenfunction problem
and such QSDs.

\begin{definition}
Let $\mu^{\varepsilon}\in\mathcal{P}([0,1)^{d})$. $\mu^{\varepsilon}$ is a QSD
for the dynamics (\ref{eqn:SD}) and the killing rate $c$ if for all bounded
and continuous functions $f$ on $[0,1)^{d}$ and $\Delta\in(0,\infty)$%
\[
\int_{\lbrack0,1)^{d}}f(x)\mu^{\varepsilon}(dx)=\frac{1}{E_{\mu^{\varepsilon}%
}e^{-\int_{0}^{\Delta}c(X_{s}^{\varepsilon})ds}}\cdot E_{\mu^{\varepsilon}%
}e^{-\int_{0}^{\Delta}c(X_{s}^{\varepsilon})ds}f(X_{\Delta}^{\varepsilon}),
\]
where $\{X_{s}^{\varepsilon},s\geq0\}$ is the solution to (\ref{eqn:SD}), and
$E_{\mu^{\varepsilon}}$ denotes that the distribution of $X_{0}^{\varepsilon}$
is $\mu^{\varepsilon}$.
\end{definition}

It follows from standard properties of nondegenerate diffusion processes that
a QSD must be absolutely continuous with respect to Lebesgue measure on
$[0,1)^{d}$ and with a smooth density, and thus we write $\mu^{\varepsilon
}(dx)=\psi^{\varepsilon}(x)dx$.

\begin{theorem}
\label{lem:qsd-ef} Assume Condition \ref{con:reg}. If $\psi^{\varepsilon}(x)$
is a QSD for (\ref{eqn:SD}) and $c$ then it is a non-negative eigenfunction
for (\ref{eqn:QSD}) with eigenvalue $\lambda^{\varepsilon}=\int_{[0,1)^{d}%
}c(x)\psi^{\varepsilon}(x)dx$. Conversely, if $(\psi^{\varepsilon}%
(x),\lambda^{\varepsilon})$ is a solution to (\ref{eqn:QSD}) with
$\lambda^{\varepsilon}>0$ and real $\psi^{\varepsilon}(x)\geq0$ with integral $1$
on $[0,1)^{d}$, then it is also a QSD for (\ref{eqn:SD}) and $c$.
\end{theorem}

The proof of this equivalence is given in the Appendix. Note that it follows
from the theorem that the two QSDs under consideration, as positive
eigenfunctions, must be unique (as QSDs).

\subsection{Nonlinear diffusion processes}

Nonlinear Markov processes arise naturally as the weak limit of many particle
systems with weak interactions between particles \cite{szn91,kolokol10}, such
as FV systems. We next define the particular NLMP that provides an alternative
stochastic representation for the solution to (\ref{eqn:QSD}). We seek a pair
$(\bar{X}_{t}^{\varepsilon},\psi_{t}^{\varepsilon}(x))$ such that the
following hold.

\medskip

\textbf{Forward process} $\bar{X}_{t}^{\varepsilon}$. Between killing events,
$\bar{X}_{t}^{\varepsilon}$ evolves according to the same dynamics as
(\ref{eqn:SDE}): $d\bar{X}^{\varepsilon}=-DV(\bar{X}^{\varepsilon}%
)dt+\sqrt{\varepsilon}dW$. For any $t>0$, let $P\{\bar{X}_{t}^{\varepsilon}\in
dx\}=\psi_{t}^{\varepsilon}(x)dx$ for $x\in\lbrack0,1]^{d}$ define $\psi
_{t}^{\varepsilon}(x)$. The final condition in the definition is that the
process is killed at rate $c(\bar{X}_{t}^{\varepsilon})$, but when killed it
is instantly reborn at a new location selected according to $\psi
_{t}^{\varepsilon}(x)dx$.

The process $\bar{X}_{t}^{\varepsilon}$ is called a \textquotedblleft
nonlinear\textquotedblright\ Markov process because, unlike $X_{t}%
^{\varepsilon}$, the evolution equation for the density $\psi_{t}%
^{\varepsilon}$ is nonlinear. The process $X_{t}^{\varepsilon}$ is killed at
rate $c(X_{t}^{\varepsilon})$, and for the convergence in \eqref{eqn:largeT}
we must renormalize by the survival probability to keep the right hand side
integrating to $1$. The rebirth mechanism of $\bar{X}_{t}^{\varepsilon}$
achieves exactly the same end, but the process survives for all time: assuming
$X_{0}^{\varepsilon}$ and $\bar{X}_{0}^{\varepsilon}$ have the same
distribution, we have
\[
P\left\{  \bar{X}_{t}^{\varepsilon}\in dx\right\}  =\frac{Ee^{-\int_{0}%
^{T}c(X_{t}^{\varepsilon})dt}1_{\{dx\}}(X_{t}^{\varepsilon})}{Ee^{-\int
_{0}^{T}c(X_{t}^{\varepsilon})dt}}.
\]
A proof of this relation can be argued by showing that both sides of the last
display are solutions to the forward equation
\[
\frac{d}{dt}\psi^{\varepsilon}_{t} (x) = \mathcal{H}^{\varepsilon
,\psi^{\varepsilon}_{t} ,\ast}\psi^{\varepsilon}_{t} (x),
\]
where
\[
\mathcal{H}^{\varepsilon,\psi,\ast}g(x)=\mathcal{L}^{\varepsilon,\ast
}g(x)+\int_{[0,1)^{d}}c(y)\psi(x)[g(y)-g(x)]dy+\int_{[0,1)^{d}}g(x)[c(y)\psi
(x)-c(x)\psi(y)]dy.
\]
It is easy to check that $\mathcal{H}^{\varepsilon,\psi^{\varepsilon},\ast
}\psi^{\varepsilon}(x)=0$, and that $\psi^{\varepsilon}(x)$ is the unique
fixed point of the forward equation.

Since the process $\bar{X}_{t}^{\varepsilon}$ is defined for all time and, as
was discussed previously $P\left\{  \bar{X}_{t}^{\varepsilon}\in dx\right\}
\rightarrow\psi^{\varepsilon}(x)dx$, it provides a natural path for an
infinite time Monte Carlo approximation. One can formulate an SDE for this
process, in which a Poisson random measure provides the killing and rebirth
mechanism. However, since it is not used we omit the precise statement.

\medskip

\textbf{Backward process} $\bar{Y}_{t}^{\varepsilon}$. Similarly there is a
NLMP $\bar{Y}_{t}^{\varepsilon}$ that follows the time reversed dynamics
(\ref{eqn:SDEdual}) between killing events, is killed at rate $\bar{c}(\bar
{Y}_{t}^{\varepsilon})$, is reborn according to $\phi_{t}^{\varepsilon
}(y)\doteq P\{\bar{Y}_{t}^{\varepsilon}\in dy\}$, and for which
\[
P\{\bar{Y}_{t}^{\varepsilon}\in dy\}\rightarrow\phi^{\varepsilon}(y)dy.
\]
Since these two processes are independent, we find
\[
P\{(\bar{X}_{t}^{\varepsilon},\bar{Y}_{t}^{\varepsilon})\in dx\times
dy\}\rightarrow\psi^{\varepsilon}(x)\phi^{\varepsilon}(y)dxdy.
\]

Let $\bar{\psi}^{\varepsilon}(x) \doteq e^{-\frac{2}{\varepsilon}V(x)}%
\phi^{\varepsilon}(x)$. Using
\begin{align*}
D\bar{\psi}^{\varepsilon}(x)=  &  -\frac{2}{\varepsilon}DV(x)e^{-\frac
{2}{\varepsilon}V(x)}\phi^{\varepsilon}(x)+e^{-\frac{2}{\varepsilon}V(x)}%
D\phi^{\varepsilon}(x)\\
D^{2}\bar{\psi}^{\varepsilon}(x)=  &  -\frac{2}{\varepsilon}D^{2}%
V(x)e^{-\frac{2}{\varepsilon}V(x)}\phi^{\varepsilon}(x)+\frac{4}%
{\varepsilon^{2}}DV(x)^{T}DV(x)e^{-\frac{2}{\varepsilon}V(x)}\phi
^{\varepsilon}(x)\\
&  -\frac{4}{\varepsilon}DV(x)^{T}e^{-\frac{2}{\varepsilon}V(x)}%
D\phi^{\varepsilon}(x)+e^{-\frac{2}{\varepsilon}V(x)}D^{2}\phi^{\varepsilon
}(x),
\end{align*}
one can verify that (\ref{eqn:QSD}) is valid for $\bar{\psi}^{\varepsilon}$,
and therefore by uniqueness of solutions establish $\psi^{\varepsilon
}(x)=e^{-\frac{2}{\varepsilon}V(x)}\phi^{\varepsilon}(x)$. The calculation is
as follows:%
\begin{align*}
-\mathcal{L}^{\varepsilon,\ast}\bar{\psi}^{\varepsilon}(x)+c(x)\bar{\psi
}^{\varepsilon}(x)=  &  -\left\langle DV(x),D\bar{\psi}^{\varepsilon
}(x)\right\rangle -\varepsilon\text{tr}[D^{2}\bar{\psi}^{\varepsilon
}(x)]-\text{tr}[D^{2}V(x)]\bar{\psi}^{\varepsilon}(x)+c(x)\psi^{\varepsilon
}(x)\\
=  &  \, \frac{2}{\varepsilon}\left\Vert DV(x)\right\Vert ^{2}e^{-\frac
{2}{\varepsilon}V(x)}\phi^{\varepsilon}(x)-\left\langle DV(x),e^{-\frac
{2}{\varepsilon}V(x)}D\phi^{\varepsilon}(x)\right\rangle \\
&  +\text{tr}[D^{2}V(x)]e^{-\frac{2}{\varepsilon}V(x)}\phi^{\varepsilon
}(x)-\frac{2}{\varepsilon}\left\Vert DV(x)\right\Vert ^{2}e^{-\frac
{2}{\varepsilon}V(x)}\phi^{\varepsilon}(x)\\
&  +2e^{-\frac{2}{\varepsilon}V(x)}\left\langle DV(x),D\phi^{\varepsilon
}(x)\right\rangle -\frac{\varepsilon}{2}e^{-\frac{2}{\varepsilon}%
V(x)}\text{tr}[D^{2}\phi^{\varepsilon}(x)]\\
&  -\text{tr}[D^{2}V(x)]e^{-\frac{2}{\varepsilon}V(x)}\phi^{\varepsilon
}(x)+c(x)e^{-\frac{2}{\varepsilon}V(x)}\phi^{\varepsilon}(x)\\
=  &  \left[  \left\langle DV(x),D\phi^{\varepsilon}(x)\right\rangle
-\frac{\varepsilon}{2}\text{tr}[D^{2}\phi^{\varepsilon}(x)]+c(x)\phi
^{\varepsilon}(x)\right]  e^{-\frac{2}{\varepsilon}V(x)}\\
=  &  \, \lambda^{\varepsilon}\phi^{\varepsilon}(x)e^{-\frac{2}{\varepsilon
}V(x)}=\lambda^{\varepsilon}\bar{\psi}^{\varepsilon}(x).
\end{align*}

Next suppose we write $\phi^{\varepsilon}(x)=e^{-\frac{1}{\varepsilon}\Phi
(x)}$ and $\psi^{\varepsilon}(x)=e^{-\frac{1}{\varepsilon}\Psi(x)}$, which is
possible with $\Phi(x)$ and $\Psi(x)$ being in $C^{2,\alpha}$. Then
$\psi^{\varepsilon}(x)=e^{-\frac{2}{\varepsilon}V(x)}\phi^{\varepsilon}(x)$
implies the relation
\[
V(x)=\frac{1}{2}[\Psi(x)-\Phi(x)]
\]
between the original dynamics and the two eigenfunctions. As we show in the
next section, this relation is the key to determining a proper swapping
mechanism. In particular, though $\Psi$ and $\Phi$ are a priori unknown, $V$
is available for the design of a scheme. Swapping is not possible for one
process in isolation, and instead one must work with backward/forward pairs.

\subsection{Swapping of a forward/backward pair}

\label{sec:swap} We follow the logic of parallel tempering to construct a jump
diffusion with the same joint stationary distribution as $(\bar{X}%
_{t}^{\varepsilon},\bar{Y}_{t}^{\varepsilon})$ while allowing swaps. For a
jump of the form
\[
(\bar{X}_{t}^{\varepsilon},\bar{Y}_{t}^{\varepsilon})=(\bar{Y}_{t-}%
^{\varepsilon},\bar{X}_{t-}^{\varepsilon})
\]
(which we refer to as a swap) we use the continuous time version of the
Metropolis rule, which means the swap occurs at rate $r(\bar{X}_{t}%
^{\varepsilon},\bar{Y}_{t}^{\varepsilon})$, where
\[
r(x,y)=e^{-\frac{1}{\varepsilon}\left(  \Psi(y)+\Phi(x)-\Psi(x)-\Phi
(y)\right)  ^{+}}%
\]
and $a^{+}=\max\{a,0\}$. Although the dynamics undergo a jump, the first
component is still following the forward dynamics and the second the backward
dynamics between jumps. The relation $V(x)=\frac{1}{2}\left[  \Psi(x)-\Phi(x)\right]  $
implies $r(x,y)=e^{-\frac{1}{\varepsilon}(2V(y)-2V(x))^{+}}$, and so this rule
can be implemented (though as noted previously the NLMP itself will require
approximation by a FV system). Let $(\bar{X}_{t}^{\varepsilon,1},\bar{Y}%
_{t}^{\varepsilon,1})$ denote the process so constructed, and note that the
pair are no longer independent. We claim that with this jump rate the process
$(\bar{X}_{t}^{\varepsilon,1},\bar{Y}_{t}^{\varepsilon,1})$ still has%
\[
\psi^{\varepsilon}(x)\phi^{\varepsilon}(y)dxdy
\]
as the large time limit of its distribution:
\begin{equation}
P\{(\bar{X}_{t}^{\varepsilon},\bar{Y}_{t}^{\varepsilon})\in dx\times
dy\}\rightarrow\psi^{\varepsilon}(x)\phi^{\varepsilon}(y)dxdy.
\label{eqn:empmconv}%
\end{equation}
This claim is a consequence of the following lemma.

\begin{lemma}
\label{lem:swap} Let $(\psi^{\varepsilon},\lambda^{\varepsilon})$
[respectively, $(\phi^{\varepsilon},\lambda^{\varepsilon})$] denote the
dominant eigenfunction/eigenvalue for $\mathcal{L}^{\varepsilon,\ast}$
[respectively, $\mathcal{L}^{\varepsilon}$]. For smooth $h:[0,1]^{d}%
\times\lbrack0,1]^{d}$ let
\begin{align*}
2\mathcal{H}^{\varepsilon}h(x,y)  &  =-\left\langle D_{x}V(x),D_{x}%
h(x,y)\right\rangle +\left\langle D_{y}V(y),D_{y}h(x,y)\right\rangle \\
&  \quad-\frac{\varepsilon}{2}\text{\emph{tr}}[D_{x}^{2}h(x,y)]-\frac
{\varepsilon}{2}\text{\emph{tr}}[D_{y}^{2}h(x,y)]\\
&  \quad-c^{\ast}(x)h(x,y)+c(x)h(x,y)+c(y)h(x,y).
\end{align*}
Then $(h^{\varepsilon},\lambda^{\varepsilon})$ is an eigenfunction/eigenvalue
pair for $\mathcal{H}^{\varepsilon}$, where $h^{\varepsilon}(x,y)=\psi
^{\varepsilon}(x)\phi^{\varepsilon}(y)$. Furthermore, if we define
\[
\mathcal{A}^{\varepsilon}f(x,y)=e^{-\frac{1}{\varepsilon}\left(  \Psi
(y)+\Phi(x)-\Psi(x)-\Phi(y)\right)  ^{+}}\left[  f(y,x)-f(x,y)\right]  ,
\]
then $(h^{\varepsilon},\lambda^{\varepsilon})$ is also the dominant
eigenfunction/eigenvalue for $\mathcal{H}^{\varepsilon}+\mathcal{A}%
^{\varepsilon,\ast}$.
\end{lemma}

\begin{proof}
The first claim follows directly from the eigenfunction properties of
$\psi^{\varepsilon}$ and $\phi^{\varepsilon}$. For the second, we note that
\begin{align*}
\mathcal{A}^{\varepsilon,\ast}h(x,y)  &  =e^{-\frac{1}{\varepsilon}\left(
\Psi(x)+\Phi(y)-\Psi(y)-\Phi(x)\right)  ^{+}}\left[  h(y,x)-h(x,y)\right] \\
&  \quad+\left[ e^{-\frac{1}{\varepsilon}\left(  \Psi(x)+\Phi
(y)-\Psi(y)-\Phi(x)\right)^{+}} -e^{-\frac{1}{\varepsilon}\left(  \Psi(y)+\Phi(x)-\Psi
(x)-\Phi(y)\right)  ^{+}}\right]  h(x,y).
\end{align*} 
Inserting $h(x,y)=\psi^{\varepsilon}(x)\phi^{\varepsilon}(y)=e^{-\frac
{1}{\varepsilon}\left(  \Psi(x)+\Phi(y)\right)  }$, we find by checking the
two cases according to the sign of~$\left(  \Psi(x)+\Phi(y)-\Psi
(y)-\Phi(x)\right)  $ that
\begin{align*}
\mathcal{A}^{\varepsilon,\ast}h(x,y)  &  =e^{-\frac{1}{\varepsilon}\left(
\Psi(x)+\Phi(y)-\Psi(y)-\Phi(x)\right)  ^{+}}e^{-\frac{1}{\varepsilon}\left(
\Psi(y)+\Phi(x)\right)  }\\
&  \quad-e^{-\frac{1}{\varepsilon}\left(  \Psi(y)+\Phi(x)-\Psi(x)-\Phi
(y)\right)  ^{+}}e^{-\frac{1}{\varepsilon}\left(  \Psi(x)+\Phi(y)\right)  }\\
&  =0.
\end{align*}
Therefore $\psi^{\varepsilon}(x)\phi^{\varepsilon}(y)$ is an eigenfunction of
$\mathcal{H}^{\varepsilon}+\mathcal{A}^{\varepsilon,\ast}$ with eigenvalue
$\lambda^{\varepsilon}$. Since $[0,1]^{d}\times\lbrack0,1]^{d}$ is irreducible
under $\mathcal{H}^{\varepsilon}+\mathcal{A}^{\varepsilon,\ast}$,
$\psi^{\varepsilon}(x)\phi^{\varepsilon}(y)$ is a strictly positive
eigenfunction and $(h^{\varepsilon},\lambda^{\varepsilon})$ is the dominant
eigenfunction/eigenvalue pair.
\end{proof}

\subsection{Infinite swapping}

\label{sec:infswap} The swap rate $r(x,y)$ of the last section can be replaced
by $Kr(x,y)$ for any $K\in(0,\infty)$ with the same conclusions, and in
particular (\ref{eqn:empmconv}) still applies. Thus a natural question is how
to pick $K$. Based on a prior analysis of parallel tempering using the large
deviation rate function to define a rate of convergence \cite{dupliupladol},
one expects that the optimal rate is only obtained in the limit $K\rightarrow
\infty$. As we will see, with a properly formulated limit process one can
\textquotedblleft hard code\textquotedblright\ the swaps, and thus the random
variables used to implement the swaps are no longer needed.

We next show how to construct the limiting process, and explain how it is used
to generate samples from the QSDs. However, before passing to the limit we
must rearrange how the samples are organized so there is a well defined limit.
Reflecting $K$ in the notation, a difficulty is that the processes
\begin{equation}
\{(\bar{X}_{t}^{\varepsilon,K},\bar{Y}_{t}^{\varepsilon,K}),\gamma
_{t}^{\varepsilon,K}(x,y),t\in\lbrack0,\infty)\}, \label{eqn:jumpdyn}%
\end{equation}
where $\gamma_{t}^{\varepsilon,K}(x,y)$ is the density of $(\bar{X}%
_{t}^{\varepsilon,K},\bar{Y}_{t}^{\varepsilon,K})$, are not tight (i.e.,
precompact), since the number of jumps of each component of (\ref{eqn:jumpdyn}%
) explodes as $K\rightarrow\infty$. However, as noted in \cite{dupliupladol}
one can avoid this problem and pass to the limit by using what amounts to
alternative bookkeeping. In this alternative formulation there are no jumps of
the type $(x,y)\rightarrow(y,x)$. Instead it is the dynamics that are swapped,
so that $\bar{X}_{t}^{\varepsilon,K}$ can, depending on the parity of the
number of jumps so far, follow either the forward dynamics or the backward
dynamics. A weak convergence analysis justifying the construction of the limit
derived below can be given using arguments analgous to those of
\cite{dupliupladol}.

We first identify the forward equation for $\gamma_{t}^{\varepsilon,K}$. The
contribution of the dynamics (but not including the killing terms) for the
pair of particles between jumps to the forward equation are handled by the
term%
\begin{align*}
2\mathcal{G}^{\varepsilon}h(x,y)  &  =-\left\langle D_{x}V(x),D_{x}%
h(x,y)\right\rangle +\left\langle D_{y}V(y),D_{y}h(x,y)\right\rangle \\
&  \quad-\frac{\varepsilon}{2}\text{tr}[D_{x}^{2}h(x,y)]-\frac{\varepsilon}%
{2}\text{tr}[D_{y}^{2}h(x,y)],
\end{align*}
which simply combines $\mathcal{L}^{\varepsilon}$ and $\mathcal{L}%
^{\varepsilon,\ast}-c^{\ast}$ appropriately. There is also a term for the
swapping dynamic. With swap rate
\[
Kr^{\varepsilon}(x,y)=Ke^{-\frac{2}{\varepsilon}(V(y)-V(x))^{+}}%
\]
for the exchange $(x,y)\rightarrow(y,x)$, the adjoint takes the form
\begin{equation}
\mathcal{G}_{s}^{\varepsilon}h(x,y)=Kr^{\varepsilon}%
(y,x)h(y,x)-Kr^{\varepsilon}(x,y)h(x,y). \label{eqn:defofGs}%
\end{equation}
Finally there are the killing and resampling terms, which in the generator of
the dynamics appear as (recall $\bar{c}=c-c^{\ast}$)
\begin{align*}
&  \int_{[0,1)^{d}}c(x)\left(  \int_{[0,1)^{d}}\gamma(z,y)dy\right)  \left[
h(z,y)-h(x,y)\right]  dz\\
&  +\int_{[0,1)^{d}}\bar{c}(y)\left(  \int_{[0,1)^{d}}\gamma(x,z)dx\right)
\left[  h(x,z)-h(x,y)\right]  dz.
\end{align*} 
The first term accounts for killing the $x$ particle at rate $c(x)$ and then
resampling from the $x$ marginal of the density $\gamma$, while keeping the
$y$ component fixed, and the reverse for the second term with rate $\bar{c}$.
The adjoints of these take the form
\begin{align*}
\mathcal{G}_{k}^{\varepsilon,\gamma}h(x,y)=  &  \int_{[0,1)^{d}}c(z)\left(
\int_{[0,1)^{d}}\gamma(x,y)dy\right)  h(z,y)dz-\int_{[0,1)^{d}}c(x)\left(
\int_{[0,1)^{d}}\gamma(z,y)dy\right)  h(x,y)dz\\
&  +\int_{[0,1)^{d}}\bar{c}(z)\left(  \int_{[0,1)^{d}}\gamma(x,y)dx\right)
h(x,z)dz-\int_{[0,1)^{d}}\bar{c}(y)\left(  \int_{[0,1)^{d}}\gamma
(x,z)dx\right)  h(x,y)dz.
\end{align*}

When combined, the forward equation for $\gamma_{t}^{\varepsilon,K}$ is%
\[
\frac{d}{dt}\gamma_{t}^{\varepsilon,K}(x,y)=\mathcal{\hat{G}}^{\varepsilon
,\gamma_{t}^{\varepsilon,K}}\gamma_{t}^{\varepsilon,K}(x,y),
\]
where
\[
\mathcal{\hat{G}}^{\varepsilon,\gamma}h(x,y)=\mathcal{G}^{\varepsilon
}h(x,y)+\mathcal{G}_{s}^{\varepsilon}h(x,y)+\mathcal{G}_{k}^{\varepsilon
,\gamma}h(x,y).
\]
As a check on the form, we verify that the fixed point of the forward equation
is $\psi^{\varepsilon}(x)\phi^{\varepsilon}(y)$. From a calculation in Lemma
\ref{lem:swap}, we know that regardless of $K$, if $h(x,y)=\psi^{\varepsilon
}(x)\phi^{\varepsilon}(y)$ then $\mathcal{G}_{s}^{\varepsilon}h(x,y)=0$. Hence
we need only check that when also $\gamma(x,y)=\psi^{\varepsilon}%
(x)\phi^{\varepsilon}(y)$, $\mathcal{G}^{\varepsilon}h(x,y)+\mathcal{G}%
_{k}^{\varepsilon,\gamma}h(x,y)=0$. Using $\int_{[0,1)^{d}}c(z)\psi
^{\varepsilon}(z)dz=\int_{[0,1)^{d}}\bar{c}(z)\phi^{\varepsilon}%
(z)dz=\lambda^{\varepsilon}$ we compute $\mathcal{G}_{k}^{\varepsilon,\gamma
}h(x,y)$ to be the sum of
\begin{align*}
&  \int_{\lbrack0,1)^{d}}c(z)\left(  \int_{[0,1)^{d}}\psi^{\varepsilon}%
(x)\phi^{\varepsilon}(y)dy\right)  \psi^{\varepsilon}(z)\phi^{\varepsilon
}(y)dz-\int_{[0,1)^{d}}c(x)\left(  \int_{[0,1)^{d}}\psi^{\varepsilon}%
(z)\phi^{\varepsilon}(y)dy\right)  \psi^{\varepsilon}(x)\phi^{\varepsilon
}(y)dz\\
&  =\int_{[0,1)^{d}}c(z)\psi^{\varepsilon}(x)\psi^{\varepsilon}(z)\phi
^{\varepsilon}(y)dz-\int_{[0,1)^{d}}c(x)\psi^{\varepsilon}(z)\psi
^{\varepsilon}(x)\phi^{\varepsilon}(y)dz\\
&  =\lambda^{\varepsilon}\psi^{\varepsilon}(x)\phi^{\varepsilon}%
(y)-c(x)\psi^{\varepsilon}(x)\phi^{\varepsilon}(y)
\end{align*}
and
\begin{align*}
&  \int_{\lbrack0,1)^{d}}\bar{c}(z)\left(  \int_{[0,1)^{d}}\psi^{\varepsilon
}(x)\phi^{\varepsilon}(y)dx\right)  \psi^{\varepsilon}(x)\phi^{\varepsilon
}(z)dz-\int_{[0,1)^{d}}\bar{c}(y)\left(  \int_{[0,1)^{d}}\psi^{\varepsilon
}(x)\phi^{\varepsilon}(z)dx\right)  \psi^{\varepsilon}(x)\phi^{\varepsilon
}(y)dz\\
&  =\lambda^{\varepsilon}\psi^{\varepsilon}(x)\phi^{\varepsilon}%
(y)-c(y)\psi^{\varepsilon}(x)\phi^{\varepsilon}(y) +c^{\ast}(x) \psi
^{\varepsilon}(x)\phi^{\varepsilon}(y).
\end{align*}
Since $\mathcal{G}^{\varepsilon}h(x,y)=\mathcal{L}_{x}^{\varepsilon,\ast}%
\psi^{\varepsilon}(x)\phi^{\varepsilon}(y)-c^{\ast}(x) \psi^{\varepsilon
}(x)\phi^{\varepsilon}(y)+\mathcal{L}_{y}^{\varepsilon}\psi^{\varepsilon
}(x)\phi^{\varepsilon}(y)$, adding and using the eigenfunction properties
indeed gives zero.

Now the limit of (\ref{eqn:defofGs}) does not make sense as $K\rightarrow
\infty$. Hence we consider the alternative formulation where the dynamics are
swapped rather than particle locations. We (somewhat severely) abuse notation
by retaining $(\bar{X}_{t}^{\varepsilon,K},\bar{Y}_{t}^{\varepsilon,K})$ for
this new process with swapped dynamics, and keep track of which component is
playing which role by adding a $\{1,2\}$-valued process $S_{t}^{\varepsilon
,K}$ that records if the assignment of dynamics is (fwd,bkwd) [$S_{t}%
^{\varepsilon,K}=1$] or (bkwd,fwd) [$S_{t}^{\varepsilon,K}=2$]. The state
space of the new process $(S_{t}^{\varepsilon,K},\bar{X}_{t}^{\varepsilon
,K},\bar{Y}_{t}^{\varepsilon,K})$ is $\{1,2\}\times\lbrack0,1)^{d}%
\times\lbrack0,1)^{d}$. The advantage of the new formulation is that a limit
is easily justified using a time scale separation argument. The $S_{t}%
^{\varepsilon,K}$ component does not converge in the usual sense, and instead
will be interpreted as a probability distribution on two points. What is key
is that the more complicated components $(\bar{X}_{t}^{\varepsilon,K},\bar
{Y}_{t}^{\varepsilon,K})$ will now have nice limits. When $K\rightarrow\infty
$, the empirical measure of $S_{t}^{\varepsilon,K}$ with respect to time
averages is trivial, and given that the particle pair $(\bar{X}_{t}%
^{\varepsilon,K},\bar{Y}_{t}^{\varepsilon,K})$ is currently at $(x,y)$,
asymptotically $S_{t}^{\varepsilon,K}$ spends the fractions of time
\begin{equation}
(F^{\varepsilon}(x,y),F^{\varepsilon}(y,x))=\frac{1}{e^{-\frac{1}{\varepsilon
}\Psi(x)}e^{-\frac{1}{\varepsilon}\Phi(y)}+e^{-\frac{1}{\varepsilon}\Psi
(y)}e^{-\frac{1}{\varepsilon}\Phi(x)}}\left(  e^{-\frac{1}{\varepsilon}%
\Psi(x)}e^{-\frac{1}{\varepsilon}\Phi(y)},e^{-\frac{1}{\varepsilon}\Psi
(y)}e^{-\frac{1}{\varepsilon}\Phi(x)}\right)  \label{eqn:defofF}%
\end{equation}
in the states $(x,y)$ and $(y,x)$, respectively, where $F^{\varepsilon}(x,y)$
is the fraction of time the particle at $x$ is in state 1 and $F^{\varepsilon
}(y,x)$ the fraction of time it is in state $2$.

We have in fact symmetrized the dynamics to make this limit possible, and the
limit dynamics of $(\bar{X}_{t}^{\varepsilon,K},\bar{Y}_{t}^{\varepsilon,K})$
with dynamics swapping rather than location swapping are defined by convex
combinations of the forward and backward dynamics. Specifially, if $\theta
_{t}^{\varepsilon,K}(x,y)$ is the density of $(\bar{X}_{t}^{\varepsilon
,K},\bar{Y}_{t}^{\varepsilon,K})$ at time $t$ then
\[
\{(\bar{X}_{t}^{\varepsilon,K},\bar{Y}_{t}^{\varepsilon,K}),\theta
_{t}^{\varepsilon,K}(x,y),t\in\lbrack0,\infty)\}
\]
tends to\footnote{We use the Skorokhod topology for sample paths of 
$[0,1)^d$-valued processes, and for densities the Dudley metric
for fixed $t$ and uniform convergence on compacts sets with respect to $t$.}
\[
\{(\bar{X}_{t}^{\varepsilon,\infty},\bar{Y}_{t}^{\varepsilon,\infty}%
),\theta_{t}^{\varepsilon,\infty}(x,y),t\in\lbrack0,\infty)\}
\]
in distribution, where
\begin{align*}
d\bar{X}_{t}^{\varepsilon,\infty} &  =\left[  F^{\varepsilon}(\bar{Y}%
_{t}^{\varepsilon,\infty},\bar{X}_{t}^{\varepsilon,\infty})-F^{\varepsilon
}(\bar{X}_{t}^{\varepsilon,\infty},\bar{Y}_{t}^{\varepsilon,\infty})\right]
DV(\bar{X}_{t}^{\varepsilon,\infty})dt+\sqrt{\varepsilon}dW_t\\
d\bar{Y}_{t}^{\varepsilon,\infty} &  =\left[  F^{\varepsilon}(\bar{X}%
_{t}^{\varepsilon,\infty},\bar{Y}_{t}^{\varepsilon,\infty})-F^{\varepsilon
}(\bar{Y}_{t}^{\varepsilon,\infty},\bar{X}_{t}^{\varepsilon,\infty})\right]
DV(\bar{Y}_{t}^{\varepsilon,\infty})dt+\sqrt{\varepsilon}d\bar{W}_t,
\end{align*}
and the $x$ component (resp., $y$ component) is killed at rate
\[
c(\bar{X}_{t}^{\varepsilon,\infty})-F^{\varepsilon}(\bar{Y}_{t}^{\varepsilon
,\infty},\bar{X}_{t}^{\varepsilon,\infty})c^{\ast}(\bar{X}_{t}^{\varepsilon
,\infty}),\quad(\text{resp., }c(\bar{Y}_{t}^{\varepsilon,\infty}%
)-F^{\varepsilon}(\bar{X}_{t}^{\varepsilon,\infty},\bar{Y}_{t}^{\varepsilon
,\infty})c^{\ast}(\bar{Y}_{t}^{\varepsilon,\infty})).
\]
If it is $\bar{X}_{t}^{\varepsilon,\infty}$ that is killed then we must decide
which state (fwd or bkwd) it was in when killed. This is done according to the
probabilities
\[
(F^{\varepsilon}(\bar{X}_{t}^{\varepsilon,\infty},\bar{Y}_{t}^{\varepsilon
,\infty}),F^{\varepsilon}(\bar{Y}_{t}^{\varepsilon,\infty},\bar{X}%
_{t}^{\varepsilon,\infty}))
\]
for (fwd,bkwd), respectively. If fwd then sample from
\[
\int_{\lbrack0,1)^{d}}\left[  F^{\varepsilon}(x,y)\theta_{t}^{\varepsilon
,\infty}(x,y)+F^{\varepsilon}(x,y)\theta_{t}^{\varepsilon,\infty}(y,x)\right]
dy
\]
to get the new location of the particle, and use
\[
\int_{\lbrack0,1)^{d}}\left[  F^{\varepsilon}(y,x)\theta_{t}^{\varepsilon
,\infty}(x,y)+F^{\varepsilon}(y,x)\theta_{t}^{\varepsilon,\infty}(y,x)\right]
dy
\]
if bkwd. To see that these are probability densities, note that for example
\begin{align*}
&  \int_{[0,1)^{d}}\int_{[0,1)^{d}}\left[  F^{\varepsilon}(x,y)\theta
_{t}^{\varepsilon,\infty}(x,y)+F^{\varepsilon}(x,y)\theta_{t}^{\varepsilon
,\infty}(y,x)\right]  dydx\\
&  =\int_{[0,1)^{d}}\int_{[0,1)^{d}}\left[  F^{\varepsilon}(x,y)\theta
_{t}^{\varepsilon,\infty}(x,y)+F^{\varepsilon}(y,x)\theta_{t}^{\varepsilon
,\infty}(x,y)\right]  dydx=1.
\end{align*}
Although the dynamics and sampling appear complicated, as we will see their
implementation with the FV approximation is not at all difficult.

The stationary density for $(\bar{X}^{\varepsilon,\infty},\bar{Y}%
^{\varepsilon,\infty})$ is
\[
h^{\varepsilon}(x,y)=\frac{1}{2}\left[  \psi^{\varepsilon}(x)\phi
^{\varepsilon}(y)+\psi^{\varepsilon}(y)\phi^{\varepsilon}(x)\right]  ,
\]
and if samples $(\bar{X}_{t}^{\varepsilon,\infty},\bar{Y}_{t}^{\varepsilon
,\infty})$ or any approximation to this process are to be used for
approximating $\psi^{\varepsilon}(x)\phi^{\varepsilon}(y)$, then the mapping
\[
F^{\varepsilon}(\bar{X}_{t}^{\varepsilon,\infty},\bar{Y}_{t}^{\varepsilon
,\infty})\delta_{(\bar{X}_{t}^{\varepsilon,\infty},\bar{Y}_{t}^{\varepsilon
,\infty})}(dx\times dy)+F^{\varepsilon}(\bar{Y}_{t}^{\varepsilon,\infty}%
,\bar{X}_{t}^{\varepsilon,\infty})\delta_{(\bar{Y}_{t}^{\varepsilon,\infty
},\bar{X}_{t}^{\varepsilon,\infty})}(dx\times dy)
\]
should be used. To check this, observe that since $x$ and $y$ appear only as
variables of integration, one has%
\[
\left[  F^{\varepsilon}(x,y)h^{\varepsilon}(x,y)+F^{\varepsilon}%
(y,x)h^{\varepsilon}(y,x)\right]  dxdy=2F^{\varepsilon}(x,y)h^{\varepsilon
}(x,y)dxdy.
\]
Using the definition of $F^{\varepsilon}$ we thus find%
\begin{align*}
2F^{\varepsilon}(x,y)h^{\varepsilon}(x,y)  &  =\frac{e^{-\frac{1}{\varepsilon
}\left[  \Psi(x)+\Phi(y)\right]  }}{e^{-\frac{1}{\varepsilon}\Psi(x)}%
e^{-\frac{1}{\varepsilon}\Phi(y)}+e^{-\frac{1}{\varepsilon}\Psi(y)}%
e^{-\frac{1}{\varepsilon}\Phi(x)}}\left[  e^{-\frac{1}{\varepsilon}\Psi
(x)}e^{-\frac{1}{\varepsilon}\Phi(y)}+e^{-\frac{1}{\varepsilon}\Psi
(y)}e^{-\frac{1}{\varepsilon}\Phi(x)}\right] \\
&  =\psi^{\varepsilon}(x)\phi^{\varepsilon}(y).
\end{align*}

\begin{remark}
\label{rem:Gibbs}We note that the method we have described applies even to the
problem of computing integrals with respect to Gibbs measures. Here the
eigenfunction problem involving $\mathcal{L}^{\varepsilon}$ has no zeroth
order term, but the one corresponding to $\mathcal{L}^{\varepsilon,\ast}$
does. Hence the forward process involves no killing, and is simply the
original dynamics, but the backward dynamics do involve killing. When using
Fleming-Viot systems to approximate these processes (before swapping), one is
therefore needed only for the backward process, and we simply simulate $N$
independent copies of the foward process. However, after coupling the systems
through swapping the impact of the killing/cloning is felt by all particles.
We note that in this special case the eigenfunction $\psi^{\varepsilon}(x)$ is
(up to normalization) the density of the Gibbs measure. It then follows from
$V(x)=\frac{1}{2}[\Psi(x)-\Phi(x)]$ and $\psi^{\varepsilon}(x)=e^{-\frac
{1}{\varepsilon}\Psi(x)}=e^{-\frac{2}{\varepsilon}V(x)}$that $\Psi(x)=2V(x)$,
$\Phi(x)=0$, and (up to normalization) $\phi^{\varepsilon}(x)=1$.

In the next section we describe the impact of coupling the forward and
backward dynamics, and give a concrete illustration in the setting of this
remark of how it speeds up the rate of convergence though an \textquotedblleft
implied potential.\textquotedblright\ 
\end{remark}

\subsection{Implied potential}

\label{sec:imppot}

An interesting demonstration of the impact of swapping that is possible when
using the INS limit is indicated by what might be called the \textquotedblleft
implied potential.\textquotedblright\ With the forward and backward processes
independent and scaling by $\varepsilon$, the potential for the pair is
\[
\left[  \Psi(x)+\Phi(y)\right]  ,
\]
so that $\psi^{\varepsilon}(x)\phi^{\varepsilon}(y)=e^{-\frac{1}{\varepsilon
}\left[  \Psi(x)+\Phi(y)\right]  }$. Although we don't know either of
$\Psi(x)$ or $\Phi(y)$ explicitly, they do define the energy landscape that is
associated with the stationary distribution $\psi^{\varepsilon}(x)\phi
^{\varepsilon}(y)dxdy$ we seek. Hence the energy barriers present in this
landscape (restricted to $[0,1]^{d}\times\lbrack0,1]^{d}$) give some idea of
the communication problems that are encountered in sampling. The dynamic for
which this is the stationary distribution are those of $(\bar{X}%
_{t}^{\varepsilon},\bar{Y}_{t}^{\varepsilon})$. One can compare the
communication issues of $(\bar{X}_{t}^{\varepsilon},\bar{Y}_{t}^{\varepsilon
})$ in navigating this potential with those of $(\bar{X}_{t}^{\varepsilon
,\infty},\bar{Y}_{t}^{\varepsilon,\infty})$ in navigating the implied
potential. As we have seen the stationary distribution of $(\bar{X}%
_{t}^{\varepsilon,\infty},\bar{Y}_{t}^{\varepsilon,\infty})$ is
\[
\frac{1}{2}\left[  \psi^{\varepsilon}(x)\phi^{\varepsilon}(y)+\psi
^{\varepsilon}(y)\phi^{\varepsilon}(x)\right]  dxdy,
\]
and so up to a constant the implied potential is
\[
-\varepsilon\log\left[  \psi^{\varepsilon}(x)\phi^{\varepsilon}(y)+\psi
^{\varepsilon}(y)\phi^{\varepsilon}(x)\right]  .
\]

Postponing a more detailed analysis of the general case to future work, we see
that in the case of Remark \ref{rem:Gibbs} the potential for the independent
and unswapped forward and backward processes is $[V(x)+1]$, which is the same
as $V(x)$, and that of the INS process is
\[
-\varepsilon\log\left[  \psi^{\varepsilon}(x)+\psi^{\varepsilon}(y)\right]
=-\varepsilon\log\left[  e^{-\frac{2}{\varepsilon}V(x)}+e^{-\frac
{2}{\varepsilon}V(y)}\right]  .
\]
When $\varepsilon>0$ is small this is approximately $W(x,y)\doteq
\min\{V(x),V(y)\}$.

\begin{example}\label{ex:periodicdiffusion_theory}
Consider the periodic diffusion
\[
dX_{t}=\sin2\pi X_{t}dt+\sqrt{2\varepsilon}dW
\]
on $[-1,1]$. 
The stationary distribution $\psi^{\varepsilon}(x)$ (note that
there is no killing/cloning in the forward direction) has two local max and
min and is proportional to $e^{-\frac{1}{2\pi\varepsilon}\cos2\pi x}$,
while $\phi^{\varepsilon}(y)=1$. 
Thus
\[
\Psi(x)=\frac{1}{2\pi}\cos2\pi x\text{ and }\Phi(y)=0
\]
and therefore
\[
V(x)=\frac{1}{2\pi}\cos2\pi x.
\]
Figure \ref{fig:1} plots $\left[  \Psi
(x)+\Phi(y)\right]  $ when $\varepsilon=0.1$.
We see there is an obvious barrier separating the two local minimum sets in Figure \ref{fig:1}.
\begin{figure}[h!]
    \centering
    \includegraphics[width=0.5\linewidth]{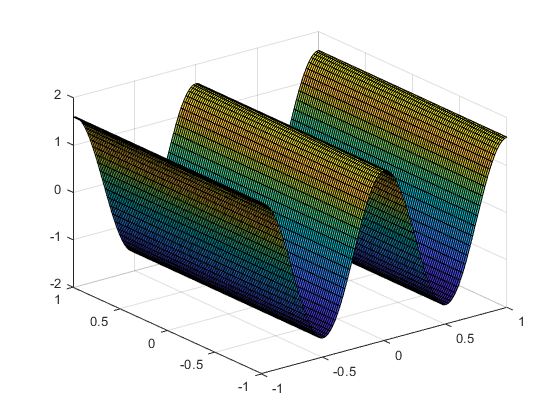}
    \caption{Original potential}
    \label{fig:1}
\end{figure}

For comparison, Figure \ref{fig:2} gives
the implied potential for the same value of $\varepsilon$.
Note that there is no apparent barrier in Figure \ref{fig:2}.
\begin{figure}[h!]
    \centering
    \includegraphics[width=0.5\linewidth]{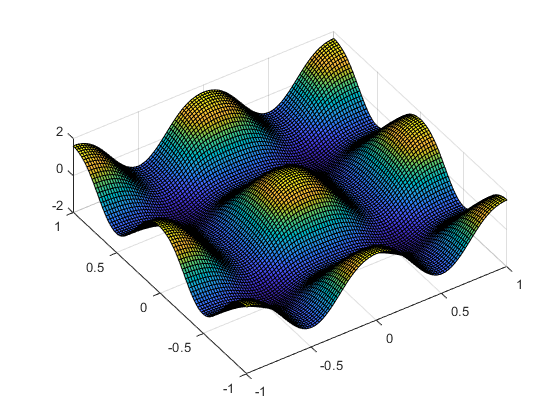}
    \caption{Implied potential}
    \label{fig:2}
\end{figure}

To examine the implied greater communication properties in more detail, we consider separately the diffusion dynamics and killing/cloning dynamics
%
of the 
symmetrized INS dynamics.
Let $(x,y)$ be a point in
$[-1,1]^{2}$. Recall that
\[
F^{\varepsilon}(x,y)=\frac{e^{-\frac{1}{\varepsilon}\left[  \Psi
(x)+\Phi(y)\right]  }}{e^{-\frac{1}{\varepsilon}\left[  \Psi(x)+\Phi
(y)\right]  }+e^{-\frac{1}{\varepsilon}\left[  \Psi(y)+\Phi(x)\right]  }}.
\]
Assume that
\[
\Psi(x)+\Phi(y)<\Psi(y)+\Phi(x)
\]
or
\[
\Psi(x)<\Psi(y).
\]
Then since $\varepsilon>0$ is small
\[
F^{\varepsilon}(x,y)\approx1\text{ and }F^{\varepsilon}(y,x)\approx0.
\]
Recall that the INS dynamics are
\begin{align*}
d\bar{X}_{t}^{\varepsilon,\infty}  &  =\left[  F^{\varepsilon}(\bar{Y}%
_{t}^{\varepsilon,\infty},\bar{X}_{t}^{\varepsilon,\infty})-F^{\varepsilon
}(\bar{X}_{t}^{\varepsilon,\infty},\bar{Y}_{t}^{\varepsilon,\infty})\right]
DV(\bar{X}_{t}^{\varepsilon,\infty})dt+\sqrt{2\varepsilon}dW\\
d\bar{Y}_{t}^{\varepsilon,\infty}  &  =\left[  F^{\varepsilon}(\bar{X}%
_{t}^{\varepsilon,\infty},\bar{Y}_{t}^{\varepsilon,\infty})-F^{\varepsilon
}(\bar{Y}_{t}^{\varepsilon,\infty},\bar{X}_{t}^{\varepsilon,\infty})\right]
DV(\bar{Y}_{t}^{\varepsilon,\infty})dt+\sqrt{2\varepsilon}d\bar{W},
\end{align*}
with $DV(\bar{X}_{t}^{\varepsilon,\infty})=\frac{1}{2}\sin2\pi x$, and the $x$
component (resp., $y$ component) is killed at rate
\[
-F^{\varepsilon}(\bar{Y}_{t}^{\varepsilon,\infty},\bar{X}_{t}^{\varepsilon
,\infty})c^{\ast}(\bar{X}_{t}^{\varepsilon,\infty}),\quad(\text{resp.,
}-F^{\varepsilon}(\bar{X}_{t}^{\varepsilon,\infty},\bar{Y}_{t}^{\varepsilon
,\infty})c^{\ast}(\bar{Y}_{t}^{\varepsilon,\infty})).
\]
So near a point with $\Psi(x)<\Psi(y)$ the dynamics are approximately
\begin{align*}
d\bar{X}_{t}^{\varepsilon,\infty}  &  =-\frac{1}{2}\sin\left(  2\pi\bar{X}%
_{t}^{\varepsilon,\infty}\right)  dt+\sqrt{2\varepsilon}dW\\
d\bar{Y}_{t}^{\varepsilon,\infty}  &  =\frac{1}{2}\sin(2\pi\bar{Y}%
_{t}^{\varepsilon,\infty})dt+\sqrt{2\varepsilon}d\bar{W},
\end{align*}
and the $x$ component (resp., $y$ component) is killed/cloned at rate
approximately
\[
0,\quad(\text{resp., }-\pi\cos(2\pi\bar{Y}_{t}^{\varepsilon,\infty})).
\]

The partition of $[-1,1]^2$ according to $F^{\varepsilon}(x,y)\approx1\text{ and }F^{\varepsilon}(y,x)\approx0$ or the reverse produces the following figure, where red denotes the drift of the diffusion when $F^{\varepsilon}(x,y)\approx1$ and blue the drift when $F^{\varepsilon}(x,y)\approx0$.

\begin{figure}[h!]
    \centering
    \includegraphics[width=0.5\linewidth]{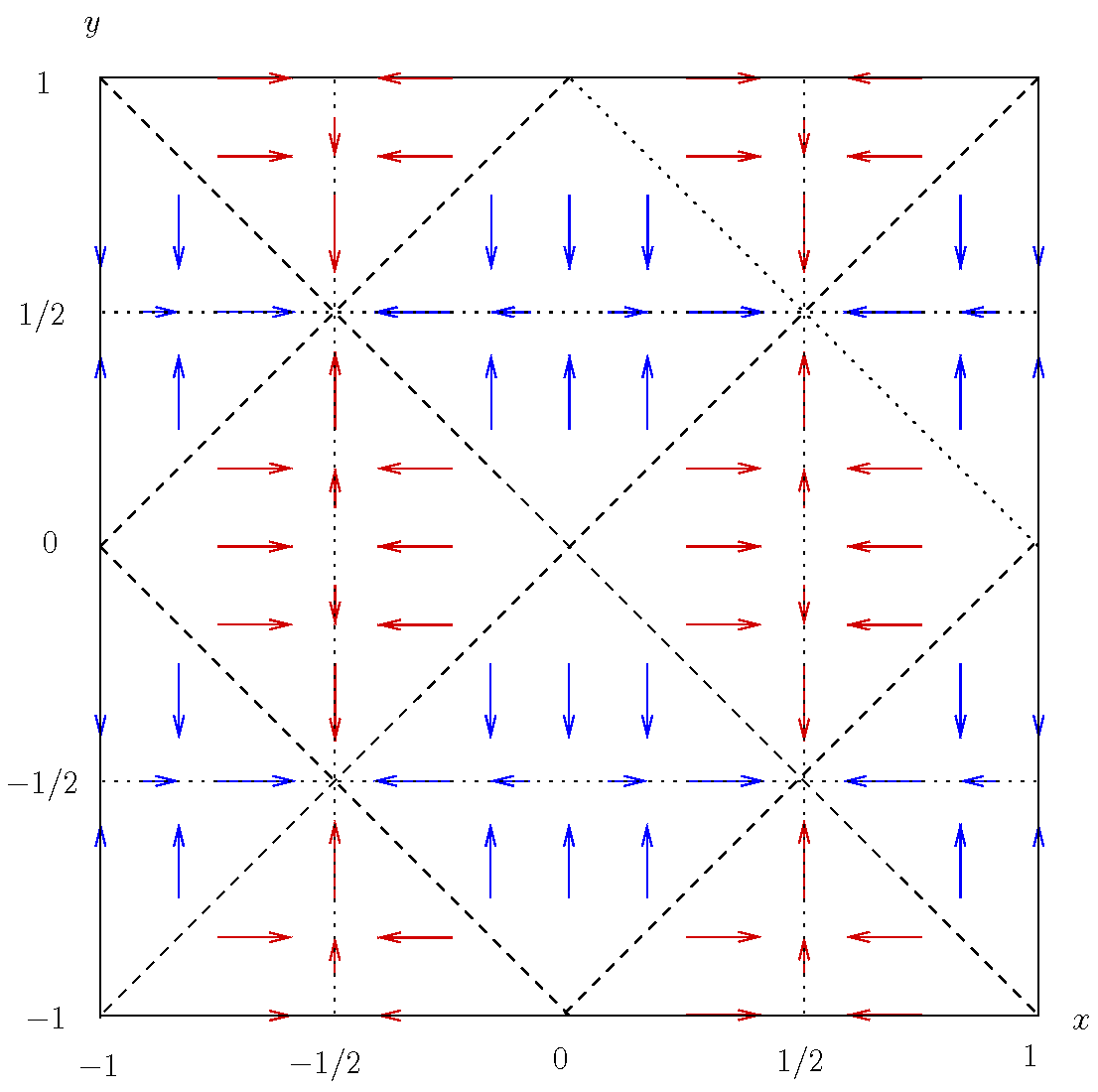}
    \caption{Diffusion drift}
    \label{fig:3}
\end{figure}

As indicated in Figure \ref{fig:2}, when $\varepsilon>0$ is small the points linking
$\{(\pm1/2,\pm1/2)\}$ are all nearly global minimum points of the implied potential.

In the absence of
killing/cloning, the diffusion dynamics act to push the process towards the
corner points $\{(\pm1/2,\pm1/2)\}$, and points of the form $\{(0,\pm1/2)\}$ or $\{(\pm1/2,0)\}$ are unstable equilibrium points of the noiseless dynamics. 
For example, if $x=-1/2$ then $\Psi(x)<\Psi(y)$ holds
for all points $y$ save $y\in\{\pm1/2\}$, where there is equality. Hence if
$-1/2<y<1/2$, then the diffusion dynamics of $\bar{X}_{t}^{\varepsilon,\infty
}$ keep the process near to $1/2$, while those of $\bar{Y}_{t}^{\varepsilon
,\infty}$ are pushing away from the unstable point $y=0$ towards $y=\pm1/2$. 
This indicates that the points $\{(\pm1/2,\pm1/2)\}$ are quasistable points under the diffusion dynamics alone. 
By comparing the original diffusion dynamics with those of the INS model,
we see that the barrier height between regions of local minima (original dynamics) and the implied barrier height for just the diffusion dynamics alone (INS dynamics) is reduced by a factor of 2.

To understand why all points on the square square connecting $\{(\pm1/2,\pm1/2)\}$ are approximate local minima of the implied potential, 
one must bring in the effect of the
killing/cloning dynamics. Suppose we consider a neighborhood of the set $\left\{
(x,y):x=-1/2,-1/2<y<1/2\right\}  $. 
Since the diffusion dynamics drive the
$y$-component away from $y=0$ and towards $y=\pm1/2$, the killing/cloning must
counteract this effect. Indeed this is exactly what occurs, since cloning,
which effectively kills the process elsewhere and shifts corresponding mass to the cloning site, is largest near
$y=0$. At the same time killing is greatest near $y=\pm1/2$, and then the mass
is redistributed according to the current distribution, to ultimately produce the (approximately uniform)
distribution on the line segment.


\end{example}

Continuing the discussion for the general case of $d=1$, one can check that if $V$ has
several local minima, then these local minima are directly linked in the INS
landscape, in that there is no energy barrier between the corresponding
two-dimensional minima. In fact a local minimum in the one dimensional setting
can correspond to an inflection point with respect to the two dimensional INS
potential. Consider for example the case where $V$ has exactly two local
minima $x_{a},x_{b}\in(0,1)$, with $V(x_{a})<V(x_{b})$, and without loss
$x_{a}<x_{b}$. Then $(x_{a},x_{a})$ is a global minimizer
of $W(x,y)=\min\{V(x),V(y)\}$. However, $(x_{b},x_{b})$ is not even a local
minimizer, since
\[
W(x_{b},y)\leq W(x_{b},x_b)\text{ for }x_{a}\leq y\leq x_{b}.
\]
Since $W(x_{b},x_{a})=W(x_{a},x_{a})$, we see it is possible to travel from
$(x_{a},x_{a})$ to $(x_{b},x_{b})$ along a path where the $W$ never exceeds
$W(x_{b},x_{b})$. 
This is a very strong statement concerning the improved communication properties of the INS process with fwd/bkwd swapping.
For comparison, the implied
potential for a two \textit{temperature} implementation of INS is calculated
in \cite{pladoldupliuwangub}, and it confirms lowered barrier heights for transitions between local
minima in comparison to no swapping. This corresponds to a reduction in
required simulation time on the order of $e^{c/\varepsilon}$ for some $c>0$.
However, the INS limit of forward/backward swapping \textit{eliminates} the
barrier for two local minima of the same value. 
This is true only for the nonlinear Markov model,
and as we will see one must also consider the impact of approximating the nonlinear problem by an interacting particle system with a finite number of particles.
In this setting,
if say all the particles are started at a point of the form $(-1/2,-1/2)$,
then to get to $(-1/2,1/2)$ or to populate a neighborhood of $(-1/2,0)$,
killing/cloning cannot play a role until one or more particles have moved away from $(-1/2,-1/2)$,
and so there is a (possibly strong) dependence on the number of particles.
Precisely quantifying the degree of improved communication,
both for the nonlinear model as well as its finite particle approximation,
is an interesting
theoretical question.

\section{Fleming-Viot approximation to the nonlinear diffusion}
\label{sec:fleming-viot}

In the last section we outlined how a coupled pair of NLMPs could be used to
define a Monte Carlo approximation to the solution of \eqref{eqn:QSD}.
However, since the dynamics of a NLMP depend on the current distribution
of the process, an approximation is needed to make implementation feasible. A
very natural approximation is in terms of what is called a Fleming-Viot
process. In this section we first describe the approximation for the forward
dynamics as defined in Section \ref{sec:nonlinearMPswap}, then indicate how
swapping will take place between the approximations for the forward and
backward processes, and lastly identify the infinite swapping limit.

\subsection{Approximation of the forward NLMP}

For convenience we first recall the definition of this process, which was
denoted $\bar{X}_{t}^{\varepsilon}$ in Section \ref{sec:nonlinearMPswap}.
Between killing events, $\bar{X}_{t}^{\varepsilon}$ evolves according to the
same dynamics as (\ref{eqn:SDE}), which means $d\bar{X}^{\varepsilon}%
=-DV(\bar{X}^{\varepsilon})dt+\sqrt{\varepsilon}dW$. For any $t>0$, let
$\psi_{t}^{\varepsilon}(x)dx=P\{\bar{X}_{t}^{\varepsilon}\in dx\}$, and recall
that all quantities are periodic and therefore without loss we can assume for
$x\in\lbrack0,1]^{d}$. Finally, the process is killed at rate $c(\bar{X}%
_{t}^{\varepsilon})$, and when killed it is instantly reborn with a new
location selected according to $\psi_{t}^{\varepsilon}(x)dx$.

The idea of the approximation is to use $N$ particles whose dynamics are
also of the form of (\ref{eqn:SDE}):%
\[
d\bar{X}_{t}^{N,n,\varepsilon}=-DV(\bar{X}_{t}^{N,n,\varepsilon}%
)dt+\sqrt{\varepsilon}dW_{t}^{(n)},\, n = 1, \ldots, N
\]
where $\{W^{(n)},n=1,\ldots,N\}$ is a collection of independent standard
$d$-dimensional Brownian motions. The processes all start with an independent
copy $\bar{X}_{0}^{N,n,\varepsilon}$ of $X_{0}^{\varepsilon}$. There is also a
collection of i.i.d. standard mean $1$ exponential random variables $\{\tau
^{(n)},n=1,\ldots,N\}$. The processes will interact through killing and
rebirth, in the following way. Letting $j$ denote the unique index for which
\[
\inf\left\{  t\geq0:\int_{0}^{t}c(\bar{X}_{s}^{N,j,\varepsilon})ds\geq
\tau^{(j)}\right\}  =\min_{1\leq n\leq N}\inf\left\{  t\geq0:\int_{0}%
^{t}c(\bar{X}_{s}^{N,n,\varepsilon})ds\geq\tau^{(n)}\right\}  ,
\]
the $j$th process is killed at time $\tau^{(j)}$. Next, an index $k$ is
selected from the uniform distribution on $\{1,\ldots,N\}$, and the particle
is then reborn with state $\bar{X}_{\tau^{(j)}-}^{N,k,\varepsilon}$, where the
minus in $\tau^{(j)}-$ indicates that the left hand limit was used. Note that
the new state of the killed particle is selected according to the empirical
distribution
\[
\frac{1}{N}\sum_{k=1}^{N}\delta_{\bar{X}_{\tau^{(j)}-}^{N,k,\varepsilon}%
}(dx).
\]
After the new state of particle $j$ has been selected, a new clock (i.e.,
exponential random variable) is generated for that particle, all other
particles' clocks are reduced by $\tau^{(j)}$, and the process repeats.

It is easy to confirm that since $c(x)$ is bounded the killings will not pile
up, and therefore the process is well defined w.p.1 for all $t\in
\lbrack0,\infty)$. The particles are in fact exchangeable, i.e., have the same
conditional distribution on path space given all the other particles
histories. Suppose we select one of the particles (say $n=1$), and consider
\[
\left(  \bar{X}_{t}^{N,1,\varepsilon},\frac{1}{N}\sum_{n=1}^{N}\delta_{\bar
{X}_{t}^{N,n,\varepsilon}}(dx)\right)  _{t=0}^{T}.
\]
As $N\rightarrow\infty$, this pair of random processes converges in
distribution to
\[
\left(  \bar{X}_{t}^{\varepsilon},\psi_{t}^{\varepsilon}(x)dx\right)
_{t=0}^{T},
\]
the NLMP and its density. 

Of course an analogous construction for the backward NLMP can be done, with
the only difference being that $DV$ is replaced by $-DV$ and $c$ by $\bar
{c}=c-c^{\ast}$ (recall that we have assumed the latter to be non-negative).

\subsection{Swapping of forward and backward FV systems}

With approximations to $\bar{X}_{t}^{\varepsilon}$ and $\bar{Y}_{t}%
^{\varepsilon}$, we can spell out how swapping is implemented. One could allow
any forward particle to be swapped with any backward particle (and such a
scheme should be examined at some point), but we will use what appears to be
the simplest for our immediate purposes. We assume that particles are paired
for life, and only the $n$th forward and backward particles can change states.
There will be as before two versions--location swapping and dynamics swapping,
and the latter will be needed for the INS limit.

To reduce notation we temporarily suppress $\varepsilon$ and $N$ in the
superscript, and let $K$ scale the swap rate. We let the location of the $N$
forward processes be denoted $\bar{X}^{(n),K}$ and the backward ones $\bar
{Y}^{(n),K}$. With swapping only allowed between $\bar{X}^{(n),K}$ and
$\bar{Y}^{(n),K}$, we have the following transitions: killing/cloning moves of
the form%
\begin{align*}
&  (\bar{X}^{(1),K},\ldots,\bar{X}^{(j),K},\ldots,\bar{X}^{(N),K}),(\bar
{Y}^{(1),K},\ldots,\bar{Y}^{(N),K})\\
&  \quad\rightarrow(\bar{X}^{(1),K},\ldots,\bar{X}^{(m),K},\ldots,\bar
{X}^{(N),K}),(\bar{Y}^{(1),K},\ldots,\bar{Y}^{(N),K})
\end{align*}
at rate $c(\bar{X}_{t}^{(j),K})$, $j\in\{1,\ldots,N\}$, where $m\in
\{1,\ldots,N\}$ is uniform and independent of all random variables used to
construct the process up till this time, and
\begin{align*}
&  (\bar{X}^{(1),K},\ldots,\bar{X}^{(N),K}),(\bar{Y}^{(1),K},\ldots,\bar
{Y}^{(k),K},\ldots,\bar{Y}^{(N),K})\\
&  \quad\rightarrow(\bar{X}^{(1),K},\ldots,\bar{X}^{(N),K}),(\bar{Y}%
^{(1),K},\ldots,\bar{Y}^{(l),K},\ldots,\bar{Y}^{(N),K})
\end{align*}
at rate $\bar{c}(\bar{Y}_{t}^{(k),K})$, $k\in\{1,\ldots,N\}$, where
$l\in\{1,\ldots,N\}$ is uniform; swapping moves of the form
\begin{align*}
&  (\bar{X}^{(1),K},\ldots,\bar{X}^{(j),K},\ldots,\bar{X}^{(N),K}),(\bar
{Y}^{(1),K},\ldots,\bar{Y}^{(j),K},\ldots,\bar{Y}^{(N),K})\\
&  \quad\rightarrow(\bar{X}^{(1),K},\ldots,\bar{Y}^{(j),K},\ldots,\bar
{X}^{(N),K}),(\bar{Y}^{(1),K},\ldots,\bar{X}^{(j),K},\ldots,\bar{Y}^{(N),K})
\end{align*}
at rate $Ke^{-(2V(\bar{Y}^{(j),K})-2V(\bar{X}^{(j),K}))^{+}}$, $j\in
\{1,\ldots,N\}$. One then uses
\[
\frac{1}{NT}\int_{0}^{T}\sum_{n=1}^{N}1_{\{dx\}}(\bar{X}_{t}^{(n),K})dt
\]
to approximate $\psi^{\varepsilon}(x)dx$ and
\[
\frac{1}{NT}\int_{0}^{T}\sum_{n=1}^{N}1_{\{dy\}}(\bar{Y}_{t}^{(n),K})dt
\]
for $\phi^{\varepsilon}(y)dy$.

\subsection{Infinite swapping limit}

As with the limit NLMPs, to take the limit $K\rightarrow\infty$ requires that
we switch from location swapping to dynamics swapping and use a variable that
records the role of each component of particle pairs. As before we abuse
notation and keep $\bar{X}_{t}^{(n),K}$ and $\bar{Y}_{t}^{(n),K}$ to designate
the components regardless of their fwd/bkwd role, and use $S_{t}^{(n),K}%
\in\{1,2\}$ for this purpose. When $S_{t}^{(n),K}=1$ $\bar{X}_{t}^{(n),K}$ is
the forward in time particle (with forward dynamics and killing rate), and
when $S_{t}^{(n),K}=2$ it is backward. Thus the FV-based estimator for
$\psi^{\varepsilon}(x)\phi^{\varepsilon}(y)dxdy$ is
\begin{align} \label{eq: empirical_distribution}
\frac{1}{NT}\int_{0}^{T}\sum_{n=1}^{N}\left[  \delta_{\{(\bar{X}_{t}%
^{(n),K},\bar{Y}_{t}^{(n),K})\}}(dx,dy)1_{\{1\}}(S_{t}^{(n),K})+\delta
_{\{(\bar{Y}_{t}^{(n),K},\bar{X}_{t}^{(n),K})\}}(dx,dy)1_{\{2\}}(S_{t}%
^{(n),K})\right]  dt.
\end{align}

This would be asymptotically unbiased, i.e., we would have the large $T$ limit
$\psi^{\varepsilon}(x)\phi^{\varepsilon}(y)dxdy$ w.p.1, if the stationary
distribution of each of the particles $(\bar{X}_{t}^{(n),K},\bar{Y}%
_{t}^{(n),K})$ was exactly $\psi^{\varepsilon}(x)\phi^{\varepsilon}(y)dxdy$.
However, there is a finite $N$ bias that vanishes as $N\rightarrow\infty$.
A formal large deviations analysis suggests that the bias decays exponentially in $N$,
and whether or not this can be made rigorous is an interesting open question.

Accepting this bias, we now send $K\rightarrow\infty$ as before. The resulting
processes and estimators take the following form. Between killings the
processes $(\bar{X}_{t}^{(n),\infty},\bar{Y}_{t}^{(n),\infty})$ follow the
dynamics
\begin{align*}
d\bar{X}_{t}^{(n),\infty}  &  =\left[  F^{\varepsilon}(\bar{Y}_{t}%
^{(n),\infty},\bar{X}_{t}^{(n),\infty})-F^{\varepsilon}(\bar{X}_{t}%
^{(n),\infty},\bar{Y}_{t}^{(n),\infty})\right]  DV(\bar{X}_{t}^{(n),\infty
})dt+\sqrt{\varepsilon}dW_{t}^{(n)}\\
d\bar{Y}_{t}^{(n),\infty}  &  =\left[  F^{\varepsilon}(\bar{X}_{t}%
^{(n),\infty},\bar{Y}_{t}^{(n),\infty})-F^{\varepsilon}(\bar{Y}_{t}%
^{(n),\infty},\bar{X}_{t}^{(n),\infty})\right]  DV(\bar{Y}_{t}^{(n),\infty
})dt+\sqrt{\varepsilon}d\bar{W}_{t}^{(n)},
\end{align*}
where $F^{\varepsilon}$ has been defined in \eqref{eqn:defofF}. The particles
are killed at rates
\[
c(\bar{X}_{t}^{(n),\infty})-F^{\varepsilon}(\bar{Y}_{t}^{(n),\infty},\bar
{X}_{t}^{(n),\infty})c^*(\bar{X}_{t}^{(n),\infty}),\quad c(\bar{Y}%
_{t}^{(n),\infty})-F^{\varepsilon}(\bar{X}_{t}^{(n),\infty},\bar{Y}%
_{t}^{(n),\infty})c^*(\bar{Y}_{t}^{(n),\infty})
\]
for $n=1,\ldots,N$. When a particle is killed we use the following rule to decide where it will be reborn.

\begin{enumerate}
\item Determine which particle from among
\[
\left\{  \bar{X}_{t-}^{(k),\infty},\bar{Y}_{t-}^{(k),\infty},k=1,\ldots
,N\right\}
\]
is to be killed. Assume for the next step that it is $\bar{X}_{t-}%
^{(n),\infty}$, and note that we follow the analogous procedure if it is
$\bar{Y}_{t-}^{(n),\infty}$.

\item Determine if $\bar{X}_{t-}^{(n),\infty}$ is in the forward or backward
state according to the probabilities
\[
(F^{\varepsilon}(\bar{X}_{t-}^{(n),\infty},\bar{Y}_{t-}^{(n),\infty
}),F^{\varepsilon}(\bar{Y}_{t-}^{(n),\infty},\bar{X}_{t-}^{(n),\infty})),
\]
for the (fwd,bkwd) assignment. Assume for the next step that it is bkwd, and
follow the analogous procedure if it is fwd.

\item With probability $1/N$ leave the particle where it is.

\item With probability $(N-1)/N$ sample from
\[
\frac{1}{N-1}\sum_{k=1,k\neq n}^{N}\left[  F^{\varepsilon}(\bar{Y}%
_{t-}^{(k),\infty},\bar{X}_{t-}^{(k),\infty})\delta_{\bar{X}_{t-}^{(k),\infty
}}(dx)+F^{\varepsilon}(\bar{X}_{t-}^{(k),\infty},\bar{Y}_{t-}^{(k),\infty
})\delta_{\bar{Y}_{t-}^{(k),\infty}}(dx)\right]
\]
to get the new location of the particle.
\end{enumerate}

Note that in step 4 we are weighting $\bar{X}_{t-}^{(k),\infty}$ by the
probability that it is a backward particle, and similarly for $\bar{Y}%
_{t}^{(n),\infty}$.

With the processes constructed, the estimator is then%
\begin{align*}
&  \frac{1}{NT}\int_{0}^{T} \sum_{n=1}^{N} \left[  F^{\varepsilon}(\bar{X}%
_{t}^{(n),\infty},\bar{Y}_{t}^{(n),\infty})\delta_{(\bar{X}_{t}^{(n),\infty
},\bar{Y}_{t}^{(n),\infty})}(dx\times dy) \right. \\
&  \mbox{} \qquad\qquad\qquad\left.  +F^{\varepsilon}(\bar{Y}_{t}^{(n),\infty
},\bar{X}_{t}^{(n),\infty})\delta_{(\bar{Y}_{t}^{(n),\infty},\bar{X}%
_{t}^{(n),\infty})}(dx\times dy)\right]  dt. \label{eqn:infiniteN}%
\end{align*}

\section{Algorithm using pure jump processes}
\label{sec:pure_jump}
 The Fleming-Viot system is a collection of interacting jump-diffusion processes, where the jumps are due to the killing and cloning of the particles. The killing and cloning processes are simple to simulate through exponential random variables.
 However, the diffusion process needs to be simulated through some numerical approximation. Diffusion processes are typically approximated through time-marching schemes, such as Euler-Maruyama in which the state is updated at deterministic time intervals. The nature of this type of approximation is quite different from jump processes where jump events occur at exponentially distributed random times. In this section we present a discretization of the infinitely swapped Fleming-Viot system where the diffusion process is approximated by a \emph{pure jump process}. This allows us to treat both killing-cloning events and the diffusion process through the same approximation scheme. 
 
\subsection{Pure jump process approximation of diffusion processes}
The pure jump process approximation of diffusion processes are based on the Markov chain method introduced by H.J. Kushner in \cite{kus77} and further developed in \cite{kusdup1}. In contrast to its standard use, however, we need not be concerned that the process take values in a grid\footnote{The Markov chain approximation method was originally motivated by approximations for solutions to linear and nonlinear PDEs based on probabilistic representations for the solution. A fixed grid was needed since the approximate solutions would often be characterized as a fixed point of a mapping on functions defined on this grid. }. Although the Markov chain method can be applied much more generally, for the
purposes of this paper we continue with just the process models relevant to
solving the problem introduced in Section \ref{sec:2}. Recall that the
generator of the process (\ref{eqn:SDE}) is%
\[
\mathcal{L}^{\varepsilon}f(x)=-\left\langle DV(x),Df(x)\right\rangle
+\frac{\varepsilon}{2}\text{tr}[D^{2}f(x)]
\]
for $f\in C_{c}^{2}(\mathbb{R}^{d})$. Following the usual recipe for
characterizing a good approximating chain, we say that a given finite
collection of vectors $\{v_{i}(x),i\in I\}$, discretization parameter $h>0$,
and a collection of jump rates $\{q_{i}(x),i\in I\}$ are \textit{locally
consistent at} $x$ if for each $f\in C_{c}^{2}(\mathbb{R}^{d})$
\begin{equation}
\label{eqn:loccons}\mathcal{L}^{\varepsilon,h}f(x)\doteq\sum_{i\in I}%
q_{i}(x)\left[  f(x+hv_{i}(x))-f(x)\right]  =\mathcal{L}^{\varepsilon}f(x)
+o(1),
\end{equation}
where $o(1)$ is a quantity $\eta^{h}(x)$ such that $\sup_{x\in\lbrack0,1]^{d}%
}\left\vert \eta^{h}(x)\right\vert \rightarrow0$ as $h\rightarrow0$. The
vectors and rates would usually be continuous in $x$, but measurability is
sufficient so long as the indicated uniformity holds for the error term.

\begin{example}
\label{exa:1}A standard construction when $v_{i}(x)$ are of the form $\left\{
\pm e_{k},k=1,\ldots,d\right\}  $, with $e_{k}$ the $k$-th unit basis vector is to use the upwinding scheme
\begin{align*}
q_{k}^{1}  &  =\frac{1}{h}(-DV(x)_{k})^{+}\text{ with }v_{k}^{1}(x)=e_{k}\\
q_{k}^{2}  &  =\frac{1}{h}(-DV(x)_{k})^{-}\text{ with }v_{k}^{2}(x)=-e_{k}%
\end{align*}
with $a^{+}=\max\{0,a\}$ and $a^{-}=-\min\{0,a\}$ to approximate the drift,
and
\begin{align*}
q_{k}^{3}  &  =\frac{\varepsilon}{2h^{2}}\text{ with }v_{k}^{3}(x)=e_{k}\\
q_{k}^{4}  &  =\frac{\varepsilon}{2h^{2}}\ \text{ with }v_{k}^{4}(x)=-e_{k}%
\end{align*}
for the Brownian motion. Indeed we have in that case (with the obvious index
set $I$)
\begin{align*}
\mathcal{L}^{\varepsilon,h}f(x)\doteq &  \sum_{j,k\in I}q_{k}^{j}(x)\left[
f(x+hv_{k}^{j}(x))-f(x)\right] \\
=  &  \,\sum_{k=1}^{d}\frac{1}{h}(-DV(x)_{k})h\frac{\partial}{\partial x_{k}%
}f(x)+\sum_{k=1}^{d}\frac{\varepsilon}{2h^{2}}h^{2}\frac{\partial^{2}%
}{\partial x_{k}^{2}}f(x)+o(1)\\
=  & \, \mathcal{L}^{\varepsilon}f(x)+o(1).
\end{align*}
We repeat that we do not use in any way that these chains take values in a
grid, and in fact the vectors and rates could depend on time and even be
random in a suitable way, so long as \eqref{eqn:loccons} holds.
\end{example}

Note that one must pay attention to the interplay between the discretization
parameter $h$ and the noise coefficient $\varepsilon$, but this would be true
with any discretization used for the purpose of sample generation. For the
numerical examples that are presented later we use the chain of Example
\ref{exa:1} since it is very simple to implement for the problems of Section
\ref{sec:2}.  

\vspace{\baselineskip} {\textbf{Dealing with $c(x)<0$.}} In the last section
it was explained how one can approximate the NLMPs by interacting particle
systems, with particles grouped into forward/backward pairs that remain linked
throughout the simulation. It was also spelled out how to construct estimators
for integrals with respect to the eigenfunctions and eigenvalue. In all the
development so far we have assumed that $c(x)\geq0$ and $\bar{c}(x) \geq0$.
Since $\mathcal{X}=[0,1)^{d}$, this could be assumed without loss if not true
by simply adding a constant to both $c$ and $\lambda^{\varepsilon}$ to make
the inequalities hold, and then removing the constant after numerical
approximation is complete. However, since the behavior of the scheme may
depend on the constant, it is also useful to know that one can deal more
directly with the case $c(x)<0$, which we now explain.

For the case $c(x)\geq0$ we can interpret the the zeroth order term in
\eqref{eqn:QSD} as a killing rate. To renormalize so that the mass associated
with the NLMP used in its representation does not decrease, the process is
reborn with state selected according to the current distribution of the NLMP.
We refer to this as ``killing and then cloning,'' since it is the killing
event that triggers the cloning event. When $c(x)<0$ in some parts of
$\mathcal{X}$, the analogous behavior is to clone particles located at such a
state $x$ with rates $-c(x)$, and then kill a particle according to the
current distribution of the NLMP. Hence we call this ``cloning and then
killing.'' One can give a precise formulation in terms of a NLMP, but for our
purposes it is enough to observe that the implementation in terms of the FV
process is obvious: for particles in locations where $c<0$ cloning events will
be generated according to this rate, after which a particle is selected for
killing according to the uniform distribution on particle indicies. Just as
with $c>0$, this must be modified according to the probability of finding
particles in fwd/bkwd states for the INS model. \vspace{\baselineskip}

We now present the final form of the algorithm, in which we allow for both
orders of killing and cloning, and also substitute the pure jump
approximations for the diffusion processes. The resulting process, in both its
finite swapping rate and infinite swapping form, are pure jump processes.
Before passing to the INS model we have the following form.  

\begin{itemize}
\item Jumps that correspond to individual particle dynamics through
approximation to the SDE.

\item Jumps that correspond to killing and then cloning or cloning and then
killing in the FV dynamics.

\item Jumps that correspond to swapping locations of a forward/backward
particle pair. This is done just as in the diffusion case of the last section.
\end{itemize}

After passing to the INS limit forward and backward particles lose their identity, the
empirical measure is replaced by a sum of two weighted empirical measures, and
the remaining dynamics are also weighted in the same way to reflect fractions
of time a given particle is acting as a forward particle versus backward.

Using Example \ref{exa:1} as the basis for the construction, we recall
$\mathcal{X}=[0,1)^{d}$ denotes the state space. For $x\in$ $\mathcal{X}$ we
define forward particle dynamic rates
\[
q_{k}^{\pm}(x)=\frac{1}{h^{2}}\left[  h(-DV(x)_{k})^{\pm}+\frac{\varepsilon
}{2}\right]
\]
and backward rates
\[
r_{k}^{\pm}(x)=\frac{1}{h^{2}}\left[  h(DV(x)_{k})^{\pm}+\frac{\varepsilon}%
{2}\right]
\]
for $k=1,\ldots,d$. Note that with both cloning and then killing and killing
and then cloning available, for $\boldsymbol{x}=(x_{1},\ldots,x_{N}%
)\in\mathcal{X}^{N}$ the generator of the forward FV system would take the
form
\begin{align*}
\mathcal{L}^{f,N,h}f(\boldsymbol{x}) &  =\sum_{n=1}^{N}\left(  \sum_{k=1}%
^{d}\sum_{\pm}q_{k}^{\pm}(x_{n})\left[  f(\boldsymbol{x}_{n,\pm}%
^{h})-f(\boldsymbol{x})\right]  \right.  \\
&  \quad\quad\left.  +1_{\{\bar{c}(x_{n})\geq0\}}\sum_{m=1}^{N}\frac{\bar
{c}(x_{n})}{N}\left[  f(\boldsymbol{x}_{n}^{m})-f(\boldsymbol{x})\right]
-1_{\{\bar{c}(x_{n})<0\}}\sum_{m=1}^{N}\frac{\bar{c}(x_{n})}{N}\left[
f(\boldsymbol{x}_{m}^{n})-f(\boldsymbol{x})\right]  \right)  ,
\end{align*}
where $\boldsymbol{x}_{n,\pm}^{h}$ is the same as $\boldsymbol{x}$ except
$x_{n}$ has been replaced by $x_{n}\pm he_{k}$ with $e_{k}$ the $k$th unit
vector in $\mathbb{R}^{d}$, and $\boldsymbol{x}_{n}^{m}$ is the same as
$\boldsymbol{x}$ except $x_{n}$ has been replaced by $x_{m}$. We also recall
that $\bar{c}(x)=$ tr$[D^{2}V(x)]+c(x)$. The construction of the backward FV
system is same but with $q_{k}^{\pm}$ replaced by $r_{k}^{\pm}$ and $\bar{c}$
replaced by $c$.

The only jump events that take place for the INS model correspond to moving
particle dynamics forward according to a convex combination of the fwd/bkwd
dyanmics, and killing/cloning events. We also recall that $F^{\varepsilon
}(x,y)$ has been defined in (\ref{eqn:defofF}), and to define the symmetric
weighted infinite swapping killing/cloning rates let $c^{\varepsilon
}(x,y)=c(x)-F^{\varepsilon}(x,y)$tr$[D^{2}V(x)]$.

A pseudocode for the algorithm obtained in the INS limit is as follows.

\begin{enumerate}
\item First decide the type of the next event. This is done by comparing the
rate for an event with a collection of independent mean one exponential random
variables. When one of the products of the rate and the elapsed time exceeds
the corresponding clock, we know the time and type of the event. The total
rate for a dynamics type event is
\begin{align*}
&  \sum_{n=1}^{N}\left(  F^{\varepsilon}(x_{n},y_{n})\sum_{k=1}^{d}\sum_{\pm
}q_{k}^{\pm}(x_{n})+F^{\varepsilon}(y_{n},x_{n})\sum_{k=1}^{d}\sum_{\pm}%
r_{k}^{\pm}(x_{n})\right. \\
&  \quad\quad\quad\left.  +F^{\varepsilon}(y_{n},x_{n})\sum_{k=1}^{d}\sum
_{\pm}r_{k}^{\pm}(y_{n})+F^{\varepsilon}(x_{n},y_{n})\sum_{k=1}^{d}\sum_{\pm
}q_{k}^{\pm}(y_{n})\right)  ,
\end{align*}
and the total rate for a killing/cloning event is
\begin{align*}
&  \sum_{n=1}^{N}\left(  c^{\varepsilon}(x_{n},y_{n})1_{\{c^{\varepsilon
}(x_{n},y_{n})\geq0\}}-c^{\varepsilon}(x_{n},y_{n})1_{\{c^{\varepsilon}%
(x_{n},y_{n})<0\}}\right. \\
&  \quad\quad\quad\left.  +c^{\varepsilon}(y_{n},x_{n})1_{\{c^{\varepsilon
}(y_{n},x_{n})\geq0\}}-c^{\varepsilon}(y_{n},x_{n})1_{\{c^{\varepsilon}%
(y_{n},x_{n})<0\}}\right)  .
\end{align*}
This identifies $N\times4\times d\times2$ types of dynamics events, and
$N\times4$ types of killing/cloning events. If it is a dynamics event, go to
Step 2, and if killing/cloning, go to Step 3.

\item Since the decomposition of dynamics events was exhaustive, we can simply
execute the indicated move. For example, if the rate that first exceeded its
corresponding clock is say $F^{\varepsilon}(y_{l},x_{l})r_{k}^{-}(x_{l})$,
then we must replace the state $x_{l}$ of particle $l$ by $x_{l}-he_{k}$. We
then subtract the elapsed time off all other clocks, generate a new clock for
the dynamics event that occurred, and return to Step 1.

\item The decomposition of killing/cloning events is not exhaustive, but does
identify which particle pair number is involved, as well as which of the two
is being killed/cloned, and whether the event is killing then cloning or
cloning then killing. For illustrative purposes we suppose the rate that first
exceeded its corresponding clock is $-c^{\varepsilon}(y_{l},x_{l}%
)1_{\{c^{\varepsilon}(y_{l},x_{l})<0\}}$. Then it is a cloning then killing
event, and the particle being cloned is $y_{l}$. We then execute the following.

\begin{enumerate}
\item With probability $1/N$ do not clone the particle.

\item With probability $(N-1)/N$ sample from
\begin{align*}
&  \frac{F^{\varepsilon}(y_{l},x_{l})}{N-1}\sum_{k=1,k\neq l}^{N}\left[
F^{\varepsilon}(x_{k},y_{k})\delta_{\{x_{k}\}}(dx)+F^{\varepsilon}(y_{k}%
,x_{k})\delta_{\{y_{k}\}}(dx)\right] \\
&  \quad+\frac{F^{\varepsilon}(x_{l},y_{l})}{N-1}\sum_{k=1,k\neq l}^{N}\left[
F^{\varepsilon}(y_{k},x_{k})\delta_{\{x_{k}\}}(dx)+F^{\varepsilon}(x_{k}%
,y_{k})\delta_{\{y_{k}\}}(dx)\right]
\end{align*}
and assign $y_{l}$ to be the new state of the selected particle. Note that
$(F^{\varepsilon}(y_{l},x_{l}),F^{\varepsilon}(x_{l},y_{l}))$ determines if
particle $y_{l}$ is determined to be in the fwd or bkwd state.

\item Subtract the time till the event off all other clocks, generate a new
clock for the killing/cloning event that did occur, and return to Step 1.
\end{enumerate}
\end{enumerate}

Having simulated trajectories of this process, the weighted empirical measure
is
\[
\frac{1}{NT}\int_{0}^{T}\sum_{n=1}^{N}\left[  F^{\varepsilon}(\bar{X}%
_{t}^{(n)},\bar{Y}_{t}^{(n)})\delta_{(\bar{X}_{t}^{(n)},\bar{Y}_{t}^{(n)}%
)}(dx\times dy)+F^{\varepsilon}(\bar{Y}_{t}^{(n)},\bar{X}_{t}^{(n)}%
)\delta_{(\bar{Y}_{t}^{(n)},\bar{X}_{t}^{(n)})}(dx\times dy)\right]  dt.
\]

\begin{remark}
\label{rem:jump} The use of a pure jump approximation to the solution of the
SDE for this problem has advantages over an Euler Maruyama approximation. One
is that it allows the other dynamics of the problem, namely the
killing/cloning events and the swapping events, to be integrated into the
approximation seamlessly, and also allows for a pure jump model for the INS
limit. A second advantage appears when applying the method to problems with
Dirichlet boundary conditions. Suppose for example that one considers the
problem of solving the eigenfunction problem on a domain $D\subset
\mathbb{R}^{d}$, with zero boundary condition. Then the implementation of the
boundary condition is very simple, in that jumps that would take the process
outside $D$ are not allowed, though an appropriate cost must be added to the
zeroth order term $c$. This property will be illustrated in future work that
focuses on hard killing.
\end{remark}

In Appendix~\ref{sec:app:pseudocode}, we provide pseudocode for all relevant algorithms for simulating a standard and infinitely--swapped Fleming-Viot system. Algorithms~\ref{alg:sdetransrate} and~\ref{alg:sdetransition} implement the pure jump approximation of the stochastic differential equation. Algorithm~\ref{alg:killclone} implements the killing and cloning events for the infinitely--swapped Fleming--Viot system, and Algorithm~\ref{alg:FVINS} simulates the infinitely--swapped Fleming--Viot system.


\section{Application to approximating solutions of ergodic control problems}

The two eigenvalue problems \eqref{eqn:QSD} and \eqref{eqn:QSDdual} are related to a particular class of ergodic control problems whose optimal controls can be expressed in terms of solutions of these eigenvalue problems. First consider the dual eigenvalue problem \eqref{eqn:QSDdual}; let $\Phi^\varepsilon(x) = - \log \phi^\varepsilon(x)$ and observe that \eqref{eqn:QSDdual} yields
\[
     - \langle D V(x), D \Phi^\varepsilon(x) \rangle e^{-\Phi^\varepsilon(x)} + \frac{\varepsilon}{2}e^{-\Phi^\varepsilon(x)}\left( \text{tr}[ D^2 \Phi^\varepsilon(x) ]- \|D \Phi^\varepsilon(x) \|^2 \right) + c(x) e^{-\Phi^\varepsilon(x)} = \lambda^\varepsilon e^{-\Phi^\varepsilon(x)} 
     \]
     and thus
     \[
      - \langle D V(x), D \Phi^\varepsilon(x) \rangle -\frac{\varepsilon}{2} \|D \Phi^\varepsilon(x) \|^2  + c(x)  + \frac{\varepsilon}{2} \text{tr}[ D^2 \Phi^\varepsilon(x) ] = \lambda^\varepsilon.
\]
The first three terms can be interpreted as a Hamiltonian
\begin{align}
    H(x,p) =  \langle -D V(x),p \rangle -\frac{\varepsilon}{2} \|p \|^2  + c(x),
\end{align}
which can be written as the convex conjugate of a running cost
\begin{align}
    H(x,p) = \min_{u\in \mathbb{R}^d} \left[ - \langle D V(x) + \sqrt{\varepsilon} u,p \rangle   + \frac{1}{2}\|u\|^2+ c(x)\right],
\end{align}
where the running cost is equal to $L(x,u) =\frac{1}{2}\|u\|^2 + c(x)$. We thus find the resulting Hamilton-Jacobi-Bellman (HJB) equation
\begin{align}
    \min_{u\in \mathbb{R}^d} \left[  \langle -D V(x) + \sqrt{\varepsilon} u,D\Phi^\varepsilon(x) \rangle   + \frac{1}{2}\|u\|^2+ c(x)\right] + \frac{\varepsilon}{2}\tr[D^2 \Phi^\varepsilon(x)] = \lambda^\varepsilon.
\end{align}
From \cite{buddupnyqwu} Theorem 2.6, this PDE is that associated with the ergodic control problem with initial condition $x$ given by 
\begin{align}
J(x) =  &\inf_{u} \limsup_{T \to \infty} \frac{1}{T} \mathbb{E}_x \int_0^T \left[\frac{1}{2} \|u_t\|^2 + c(X_t) \right] dt  
\end{align}
subject to 
\[
 dX_t = \left[ -DV(X_t) + \sqrt{\varepsilon}u_t \right] dt + \sqrt{\varepsilon}dW_t, \, X_0 = x. \nonumber
\]
Here the infimum is over a suitable collection of processes $\{u_t,t\geq 0\}$
that are progressively measurable with respect to a filtration that measures the Wiener process.
Moreover, the same theorem states that an optimal time independent feedback control is characterized through the solution of the PDE: $u(x) = -\sqrt{\varepsilon} D \Phi^\varepsilon(x) = \sqrt{\varepsilon}D \log \phi^\varepsilon(x)$, and the eigenvalue provides the optimal cost, $J(x) = \lambda^\varepsilon$ for all initial conditions $x$. 

The eigenvalue problem \eqref{eqn:QSD} yields a similar control problem. One can show that  $\Psi^\varepsilon = -\log \psi^\varepsilon$ satisfies the HJB equation
\begin{align}
    \min_{u\in \mathbb{R}^d} \left[  \langle D V(x) + \sqrt{\varepsilon} u,D\Psi^\varepsilon(x) \rangle   + \frac{1}{2}\|u\|^2+ c(x) + c^\ast(x)\right] + \frac{\varepsilon}{2}\tr[D^2 \Psi^\varepsilon(x)] = \lambda^\varepsilon,
\end{align}
which corresponds to ergodic control problem 
\begin{align}
    J(y) =  &\inf_{u} \limsup_{T \to \infty} \frac{1}{T} \mathbb{E}_y \int_0^T \left[\frac{1}{2} \|u_t\|^2 + c(Y_t) + c^\ast(Y_t) \right] dt 
  \end{align}  
  subject to 
\[
  dY_t = \left[ DV(Y_t) + \sqrt{\varepsilon}u_t \right] dt + \sqrt{\varepsilon}dW_t, \, Y_0 = y. 
  \]
Also analogously, an optimal feedback control is $u(y) = -\sqrt{\varepsilon} D\Psi^\varepsilon(y) = \sqrt{\varepsilon}D\log \psi^\varepsilon(y)$ and the optimal cost is $J(y) = \lambda^\varepsilon$.   

The interacting particle system outlined in this paper produces samples from quasistationary distributions $\psi^\varepsilon(x)$ and $\phi^\varepsilon(x)$. The ergodic control, however, requires evaluation of the densities, in particular, the \emph{score function}. Given only samples from the QSDs, we can apply score-based generative modeling (SGM) tools to estimate the control function. From SGMs, we parametrize a neural net $\mathsf{s}_\theta: \mathbb{R}^d \to \mathbb{R}^d$ and train it through regression against the optimal control
\begin{align}
    \min_\theta \int_{[0,1)^d} \|\mathsf{s}_\theta(x) - D\log \phi^\varepsilon(x) \|^2 \phi^\varepsilon(x) dx.
\end{align}
This least-squares objective cannot be computed in practice. There is, however, an equivalent objective that requires only evaluations of $\mathsf{s}_\theta$ on samples of $\phi^\varepsilon(x)$. Through integration by parts, the ergodic control can be approximated through minimizing
\begin{align}
    \min_\theta \int_{[0,1)^d} \left[ \|\mathsf{s}_\theta(x) \|^2  + 2 \text{div}( \mathsf{s}_\theta(x)) \right] \phi^\varepsilon(x)  dx.
\end{align}
Score-based generative modeling has been applied in high-dimensional settings, suggesting that high-dimensional ergodic control problems can be addressed through a combination of the interacting particle system and score-matching tools.


\section{Numerical examples: Eigenfunctions and quasistationary distributions}

Code relating to the numerical examples here can be found at \\ \texttt{github.com/benjzhang/ergodic\_control\_qsd}. 

\subsection{Sampling from multimodal distributions}
A special case of the eigenvalue problems \eqref{eqn:QSD} and \eqref{eqn:QSDdual} is the case where $\lambda^\varepsilon = 0$, in which case the infinitely swapped Fleming-Viot particle system should approximate the stationary distributions of the Langevin dynamics given by the forward system.

\begin{example} \label{ex:periodicdiffusion_computation}
We revisit the periodic diffusion Example \ref{ex:periodicdiffusion_theory} to empirically evaluate and explore the properties of our forward-backward infinite swapping Fleming-Viot method. We are interested in producing samples from and computing expectations with respect to the Gibbs measure 
\begin{align}
    \rho^\varepsilon(x) \propto \exp\left(- \frac{1}{2\pi\varepsilon} \cos 2\pi x \right), \, x \in [-1,1].
\end{align}
Then $V(x) = \frac{1}{2\pi}\cos(2\pi x)$, and we have the pair of differential equations 
\begin{align}
 \sin(2\pi x) D\psi^\varepsilon(x) - \varepsilon D^2 \psi^\varepsilon(x) + 2\pi \cos(2\pi x) \psi^\varepsilon(x)= 0 \\
-\sin(2 \pi x) D\phi^\varepsilon(x) - {\varepsilon} D^2 \phi^\varepsilon(x) = 0
\end{align}
with corresponding dynamical representations 
\begin{align}
    &dX^\varepsilon(t) = \sin(2\pi X^\varepsilon(t)) dt + \sqrt{2\varepsilon} dW(t), \, c(x) = 0 \\
    &dY^\varepsilon(t) = - \sin(2\pi Y^\varepsilon(t)) dt + \sqrt{2\varepsilon} d\bar{W}(t), \, \bar{c}(y) = \pi  \cos (2\pi y)
\end{align}
Notice that there is no killing or cloning for the forward particle $X^\varepsilon(t)$ while the correction term yields a killing or cloning rate related to the Laplacian of the potential function for the backward particle $Y^\varepsilon(t)$. 

In Figure~\ref{fig:periodicdiff_loweps_lowN}, we see that with low temperatures ($\varepsilon = 0.05$), neither the Langevin dynamics nor the infinitely-swapped Fleming-Viot system are effective at exploring the potential. In the infinite particle limit, the effective potential of the INS system should nearly eliminate the barriers between metastable states. However, with a small number of particles ($N = 5$) the particles exhibit ``herding,'' in which the particles tend to cluster together and prevent any individual particle from effectively exploring the state space.

Herding can be addressed with more particles, as Figure~\ref{fig:periodicdiff_loweps_bigN}, where particles have a better chance of exploring the effective potential, or with higher temperatures, as Figure~\ref{fig:periodicdiff_bigeps_lowN}. The tradeoff of using more particles is that more events will need to be simulated, thereby increasing computational cost, while using a higher temperature does not solve the original problem. 
An alternative to simply increasing $N$ is to, in the spirit of parallel tempering, exchange particle pairs of the INS system with different temperatures.
This option will be considered in future work.


\begin{figure} 
    \centering
    \begin{subfigure}{0.4\textwidth}
    \includegraphics[trim={35pt 11pt 55pt 17pt},clip,width=\linewidth]{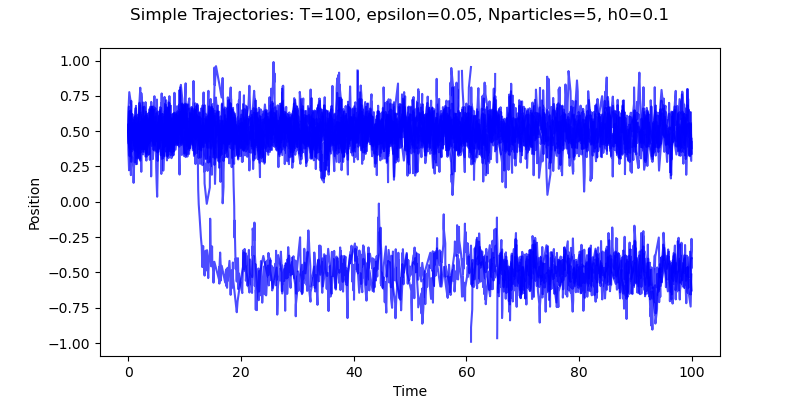}
     \includegraphics[trim={35pt 11pt 55pt 17pt},clip,width=\linewidth]{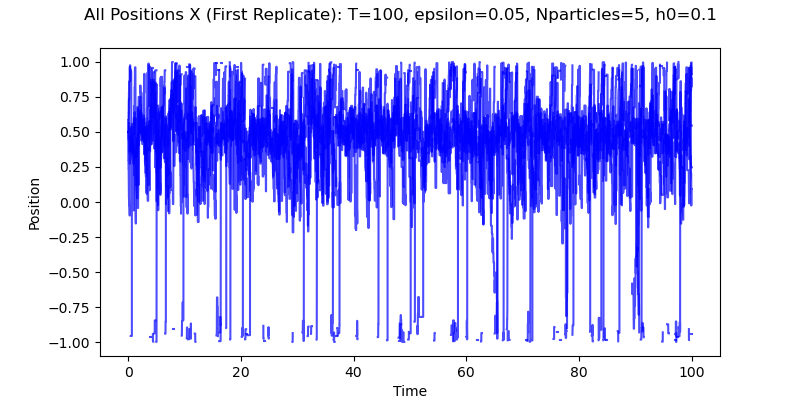}
    \caption{Independent Langevin trajectories (top) do not mix well between metastable states. INS Fleming-Viot system (bottom) with few particles at low temperatures exhibits herding. }
    \end{subfigure}
      \begin{subfigure}{0.58\textwidth}
    \includegraphics[trim={0 25pt 130pt 48pt},clip,width=\linewidth]{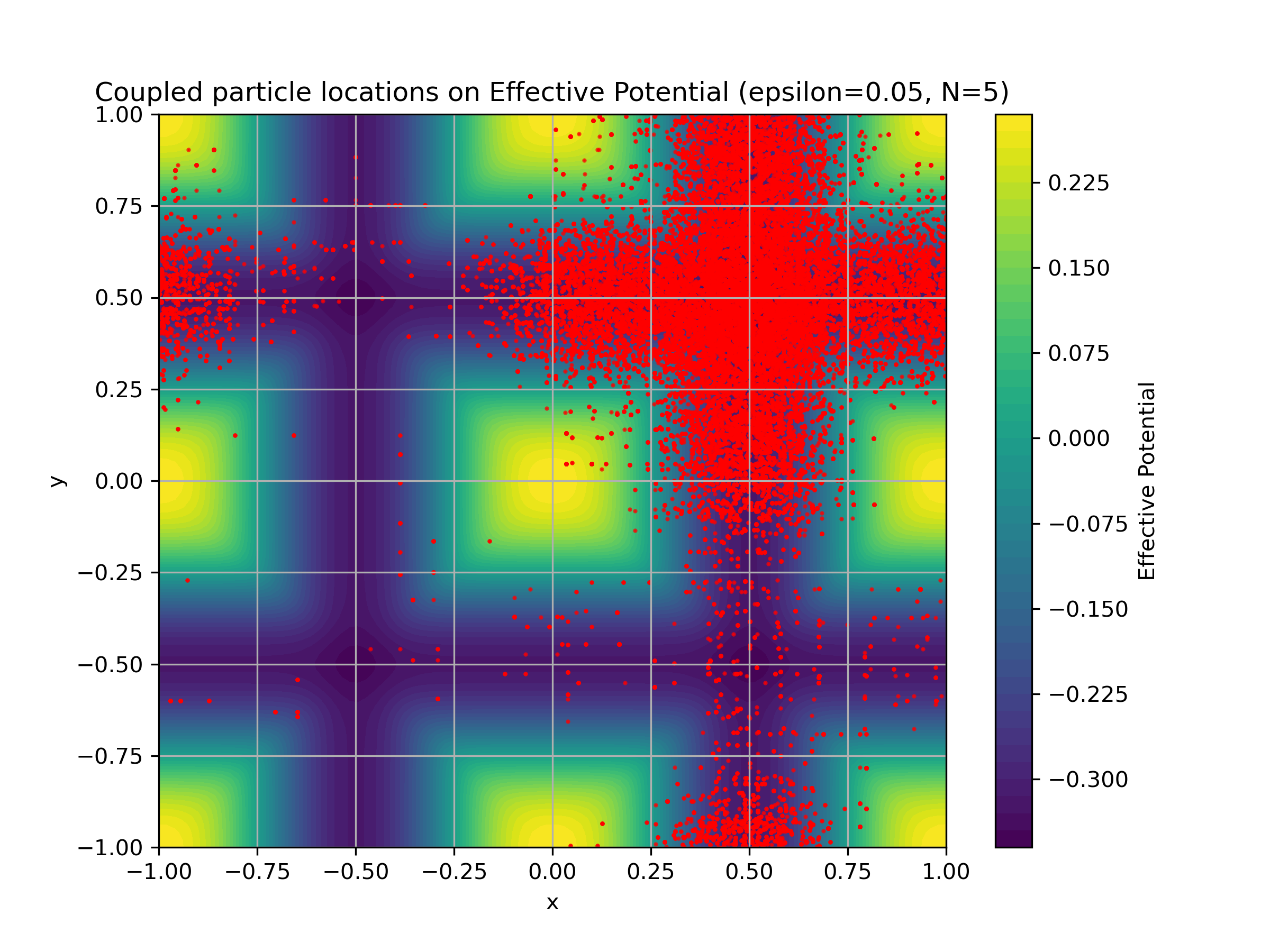}
    \caption{Herding can also be seen on the effective potential for the paired particles. Particle locations are show up to time $T = 100$. }
    \end{subfigure}
    \caption{ Periodic diffusion with $N = 5$ and $\varepsilon = 0.05$. At low temperatures, the particle system does not mix well due to lack of mixing between metastable states and herding in the INS Fleming-Viot system.   }
    \label{fig:periodicdiff_loweps_lowN}
\end{figure}

\begin{figure} 
    \centering
    \begin{subfigure}{0.4\textwidth}
    \includegraphics[trim={35pt 11pt 55pt 17pt},clip,width=\linewidth]{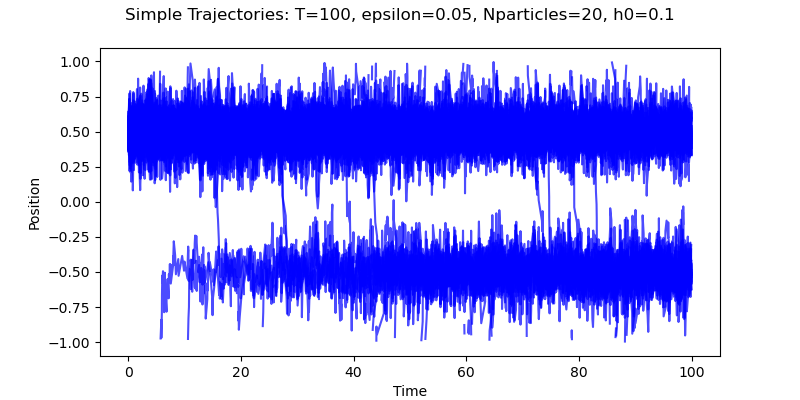}
     \includegraphics[trim={35pt 11pt 55pt 17pt},clip,width=\linewidth]{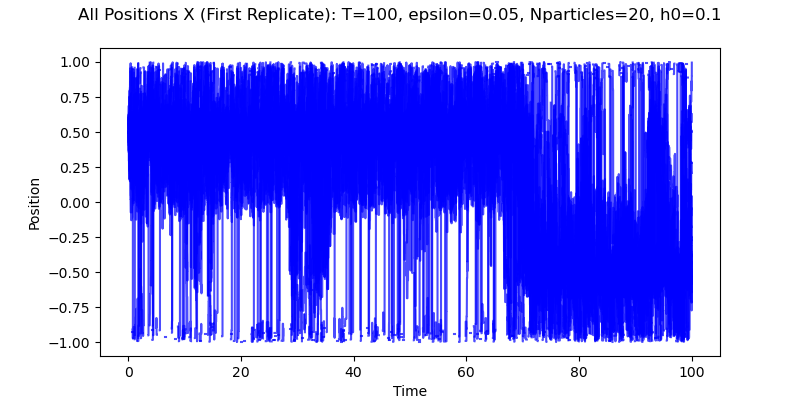}
    \caption{With more particles, independent Langevin trajectories (top) still struggle to move between metastable states, but the INS Fleming-Viot system (bottom), while still exhibiting some herding, explores the state space better.   }
    \end{subfigure}
      \begin{subfigure}{0.58\textwidth}
    \includegraphics[trim={0 25pt 130pt 48pt},clip,width=\linewidth]{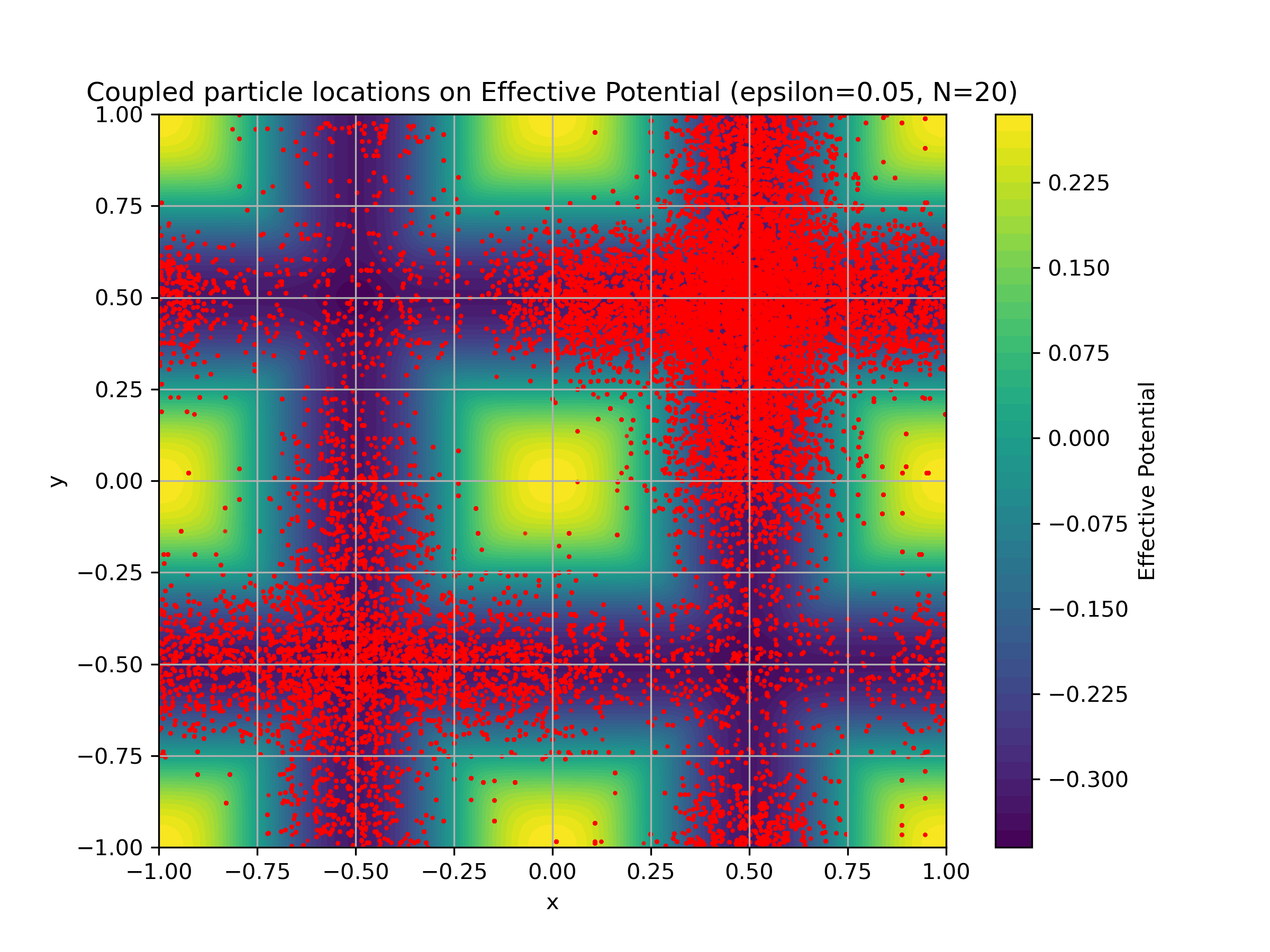}
    \caption{With higher temperatures, herding can still be seen on the effective potential, but the particle system is no longer trapped in one mode.  }
    \end{subfigure}
    \caption{ Periodic diffusion with $N = 20$ and $\varepsilon = 0.05$. With more particles, the INS Fleming-Viot system overcomes herding.    }\label{fig:periodicdiff_loweps_bigN}
\end{figure}

\begin{figure} 
    \centering
    \begin{subfigure}{0.4\textwidth}
    \includegraphics[trim={35pt 11pt 55pt 17pt},clip,width=\linewidth]{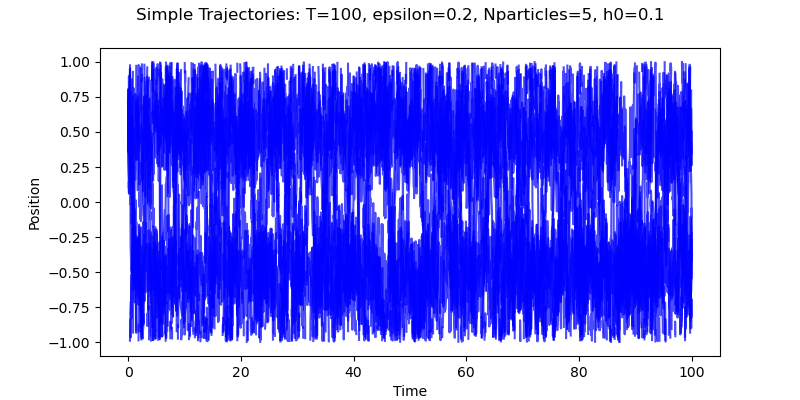}
     \includegraphics[trim={35pt 11pt 55pt 17pt},clip,width=\linewidth]{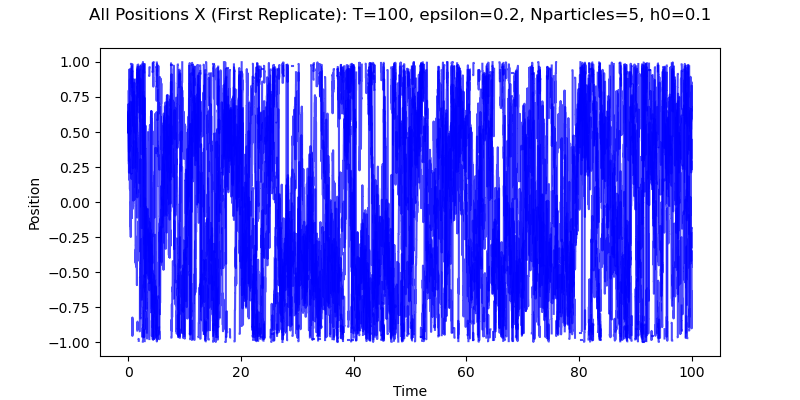}
    \caption{With high temperatures, independent Langevin trajectories (top) can move better between metastable states, but the INS Fleming-Viot system (bottom) explores the state space better.   }
    \end{subfigure}
      \begin{subfigure}{0.58\textwidth}
    \includegraphics[trim={0 25pt 130pt 48pt},clip,width=\linewidth]{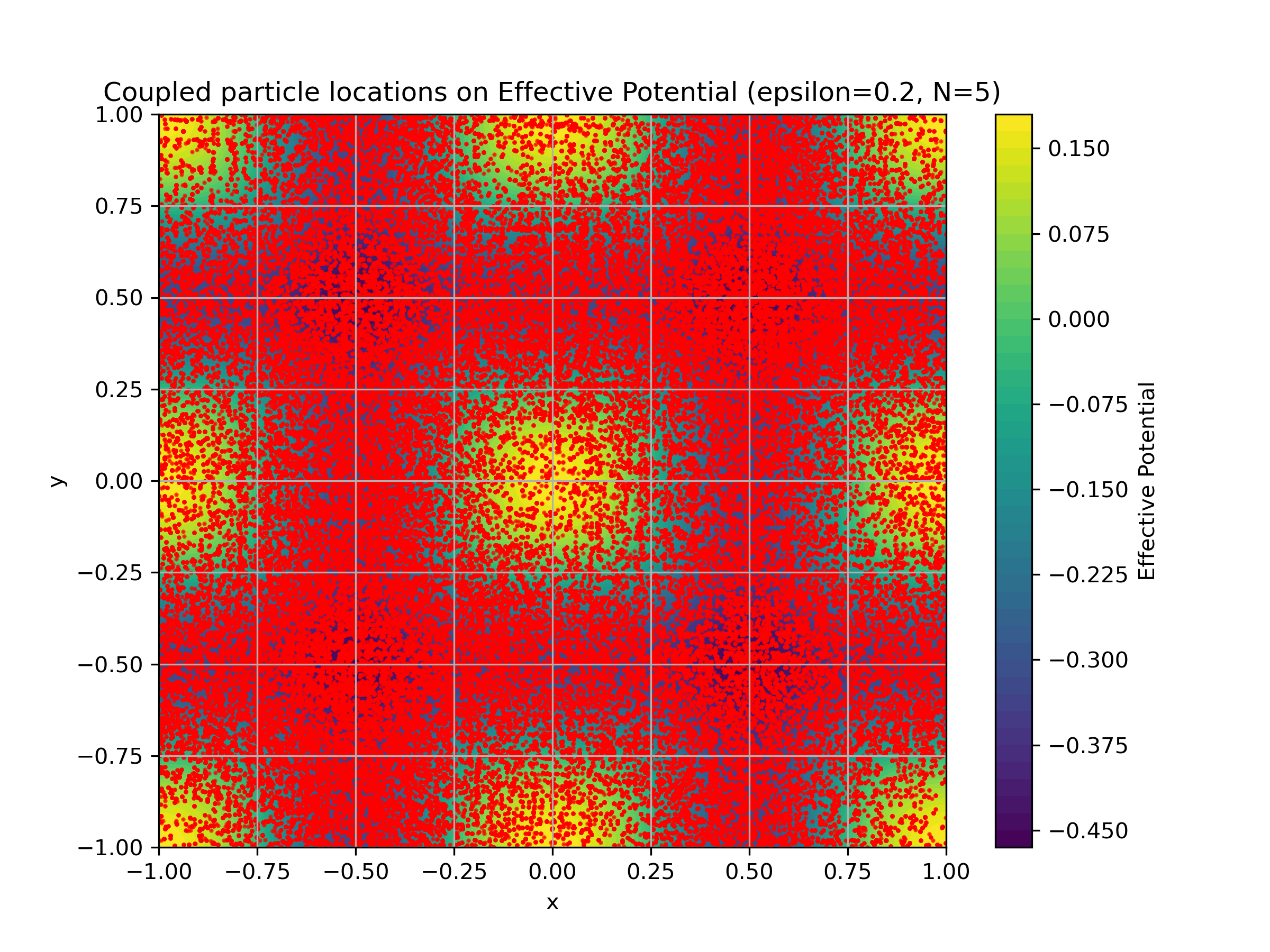}
    \caption{The particle system no longer herds and can effective explore the effective potential.  }
    \end{subfigure}
    \caption{ Periodic diffusion with $N = 5$ and $\varepsilon = 0.2$. With higher temperatures, the INS Fleming-Viot system also overcomes herding.    }\label{fig:periodicdiff_bigeps_lowN}
\end{figure}

\end{example}

\begin{example}[Mixture of Gaussians]

Consider the potential on the periodic domain $[0,4]\times[0,4]$ that consists of an array of individual Gaussians distributions
\begin{align}
    \pi(x_1,x_2) \propto \sum_{n = 1}^K \sum_{m = 1}^K \exp\left(- \frac{(x_1 - n)^2 + (x_2 - m)^2}{2\sigma^2} \right)
\end{align}
for $K = 4$. The Langevin dynamics corresponding to this distribution are 
\begin{align}
    dX(t) = \nabla \log \pi(X(t)) dt + \sqrt{2 \varepsilon} dW(t).
\end{align}
In our numerical examples, we present results for $\varepsilon = 0.4$ and $\sigma = 0.1$. We approximate the INS particle system with $N = 5$ pairs of particles over a time horizon of $T = 100$. The diffusion dynamics are approximated by the pure jump process with a randomized grid where $h$ is iid and distributed according to $\mathcal{U}(0.05,0.15)$. For comparison, we also simulate a standard Fleming-Viot system with $N = 10$ particles over the same time horizon. Note that as there is no killing or cloning, the particle system is equivalent to an ensemble of $N = 10$ independent Langevin trajectories. 

In Figure~\ref{fig: ensemble_comparison} we plot the INS and Fleming-Viot systems up to $T = 25$. Notice that the INS particle system effectively explores the entire state space with this time horizon, while the standard Fleming-Viot system mixes poorly and only discoveres a small number of modes. 

The INS ensemble in Figure~\ref{fig: ensemble_comparison} does not properly represent the empirical measure as the samples are weighted according to the infinite swapping ratio. In Figure~\ref{fig: sample_comparison} we resample from a single trajectory in the INS particle system according to the empirical distribution \eqref{eq: empirical_distribution} to better visualize the particle approximation of the target distribution. For comparison, we plot true samples of the array of Gaussians as well. 

\begin{figure}
    \centering
    \begin{subfigure}{0.45\textwidth}
    \includegraphics[trim={0 0 75pt 17pt},clip,width=\linewidth]{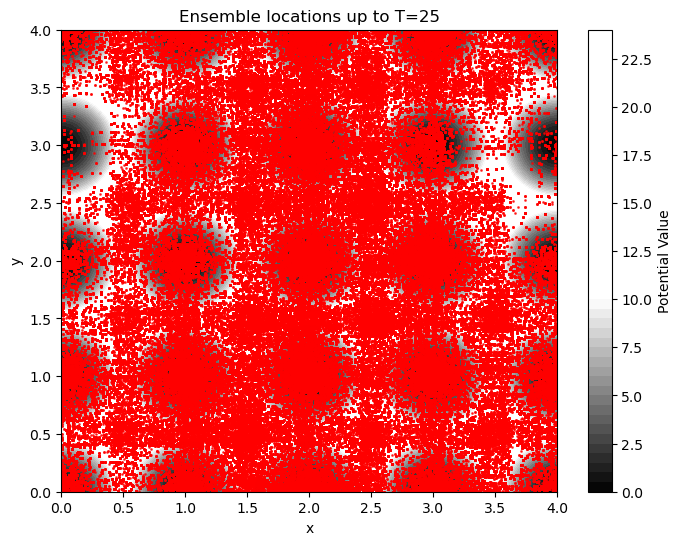}
    \caption{The INS ensemble explores the state space effectively.}
    \end{subfigure}
    ~
    \begin{subfigure}{0.45\textwidth}
    \includegraphics[trim={0 0 75pt 17pt},clip,width=\linewidth]{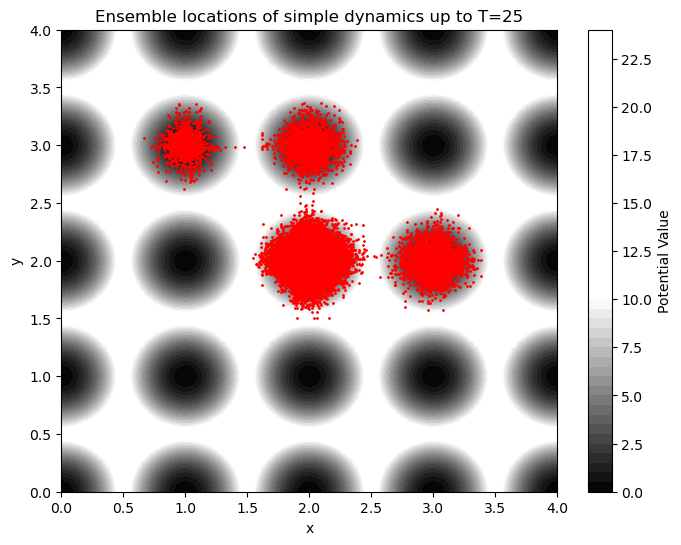}
    \caption{Multimodality is a challenge for an ensemble of independent Langevin trajectories. }
    \end{subfigure}
    \caption{ INS and standard Fleming-Viot particle systems trace plots up to $T = 25$. Each system consists of $N = 5$ particles.   }
    \label{fig: ensemble_comparison}
\end{figure}

\begin{figure}[h]
    \centering
    \begin{subfigure}{0.45\textwidth}
    \includegraphics[trim={0 0 75pt 17pt},clip,width=\linewidth]{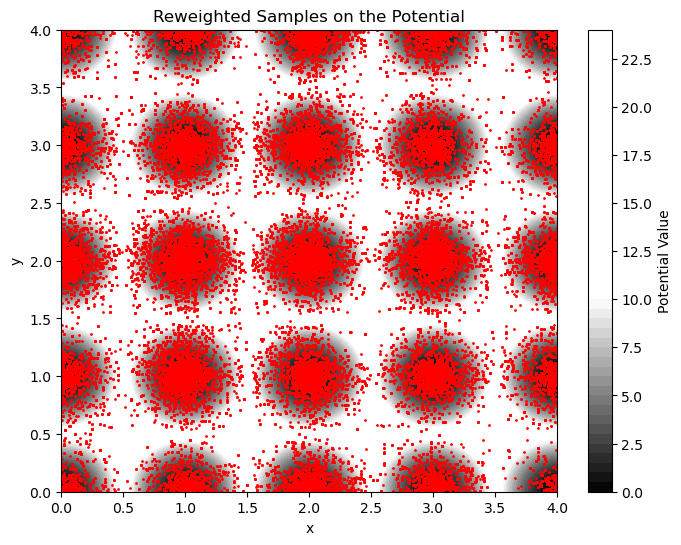}
    \caption{Reweighted ensemble removes particles in low density regions.  }
    \end{subfigure}
    ~
    \begin{subfigure}{0.45\textwidth}
    \includegraphics[trim={0 0 75pt 17pt},clip,width=\linewidth]{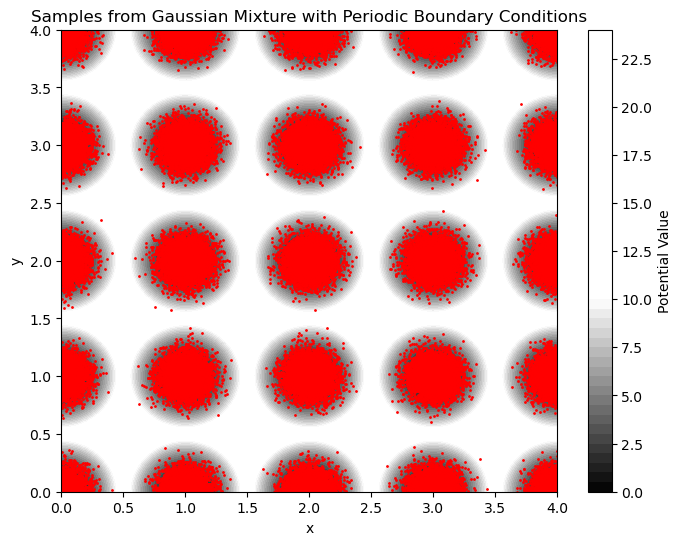}
    \caption{Exact samples from the array of Gaussians. \text{   }}
    \end{subfigure}
    \caption{A single INS trajectory of length $T = 100$ is resampled according to its empirical measure. Resampling removes particles in low density regions. }
    \label{fig: sample_comparison}
\end{figure}

\end{example}



\subsection{Quasistationary distributions}
We consider a variation of the periodic diffusion Example~\ref{ex:periodicdiffusion_theory} where dynamics are extended to the entire real line and there is nontrivial killing function $c(x)$ to prevent the particles from escaping to infinity. We wish to solve the eigenvalue problem 
\begin{align} \label{eq:sincosinf}
    -\sin(2\pi x) D\phi^\varepsilon(x) - \varepsilon D^2 \phi^\varepsilon(x) + c(x) \phi^\varepsilon(x) = \lambda^\varepsilon \phi^\varepsilon(x), \, \text{ in } x \in \R,
\end{align}
where $c(x) = x^2/\pi$. As the true solution is difficult to express analytically, we validate the stochastic representation of the eigenfunction against a finite difference approximation of the eigenvalue problem. 

In our example, we choose $\varepsilon = 0.125$. We choose to discretize the eigenvalue problem over the interval $x \in [-4,4]$ with grid size $\Delta x = 0.008$. We find that the eigenvalue approximated with the finite difference scheme is $\lambda^\varepsilon = 0.143$. We plot the eigenfunction in Figure~\ref{fig:fdsol}. The eigenfunction is normalized to integrate to one. As the eigenfunction is expected to be a density function (the quasistationary distribution), we consider the eigenvalue problem restricted to a domain large enough to capture the bulk of the QSD. 

\begin{figure}[h]
    \centering
\includegraphics[width=0.5\linewidth]{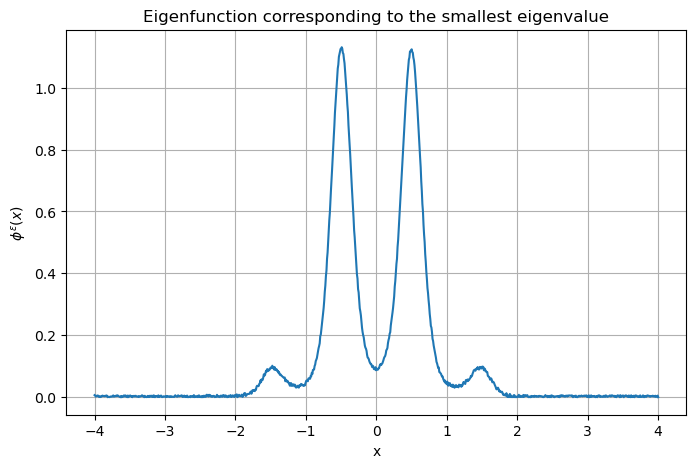}
    \caption{Eigenfunction for eigenvalue problem \eqref{eq:sincosinf}. The eigenvalue is approximately $\lambda^\varepsilon = 0.143$. }
    \label{fig:fdsol}
\end{figure}

For the FV and INS particle systems, we choose $N = 50$ particles, average SDE jump size $h_0 = 0.1$. The first 10 seconds of the simulation is discarded as burn-in. In Figure~\ref{fig: qsd_ensemble_comparison}, we show and compare trace plots of the INS and Fleming-Viot particle systems. Once again, observe that in low noise regimes, the vanilla FV system do not mix well, while the INS system is able to effectively explore the state space.

\begin{figure}
    \centering
    \begin{subfigure}{0.45\textwidth}

    \includegraphics[width = \textwidth]{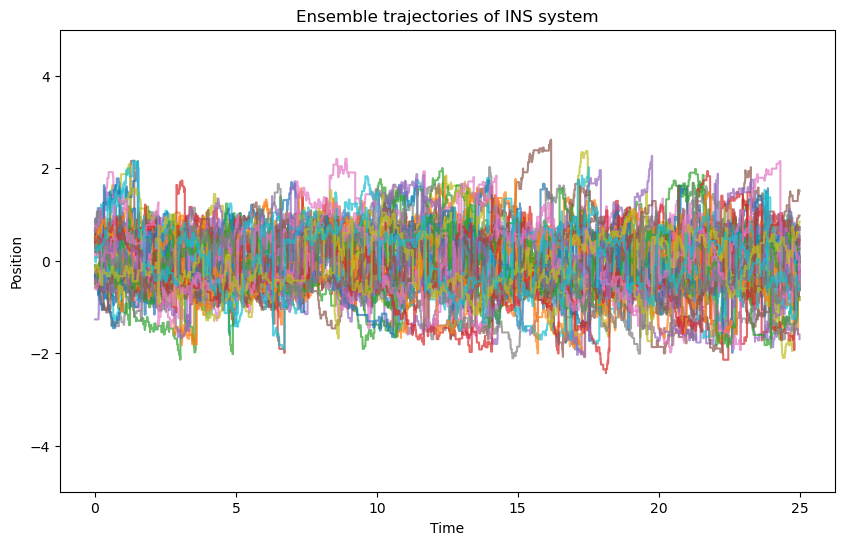}
    \caption{The INS ensemble explores the state space effectively.}
    \end{subfigure}
    ~
    \begin{subfigure}{0.45\textwidth}
    \includegraphics[width = \textwidth]{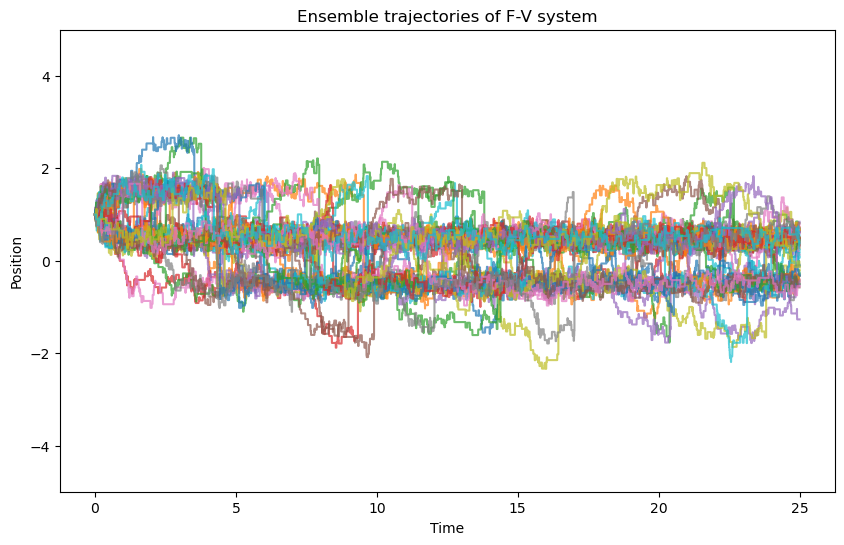}
    \caption{Multimodality is a challenge for standard Fleming-Viot ensembles. Individual trajectories can remain in modes for long periods of time. }
    \end{subfigure}
    \caption{ INS and standard Fleming-Viot particle systems trace plots up to $T = 25$ for the eigenvalue problem in \eqref{eq:sincosinf}. Each system consists of $N = 50$ particles.   }
    \label{fig: qsd_ensemble_comparison}
\end{figure}

Both particles estimate the true eigenvalue decently well. In Figure~\ref{fig:eigvalmse}, we plot the mean-squared error of the two particle systems and show that the INS system has a lower MSE than the simple FV system, but the improvement is relatively marginal.  
\begin{figure}
    \centering
    \includegraphics[width=0.5\linewidth]{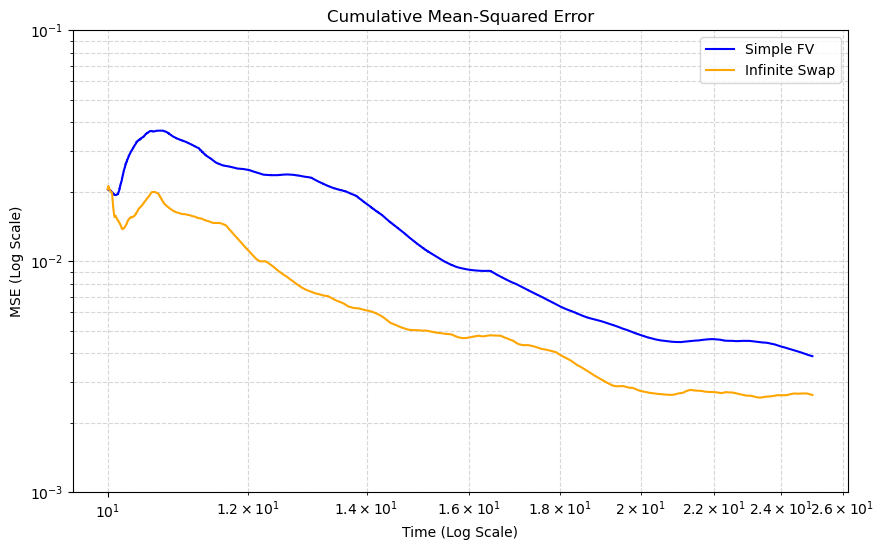}
    \caption{Cumulative mean-squared error as a function of time for estimating the eigenvalue $\lambda^\varepsilon$ using the FV and INS systems. }
    \label{fig:eigvalmse}
\end{figure}

The INS system, however, produces a better representation of the true quasistationary distribution. In Figure~\ref{fig: qsdeig}, we plot the resulting histograms of the INS and standard Fleming-Viot system that approximate the quasi-stationary distribution. We overlay the finite difference approximation of the true eigenfunction on the histograms. Observe that the INS histogram more accurately approximates the true QSD.  

\begin{figure}
    \centering
    \begin{subfigure}{0.45\textwidth}
    \includegraphics[width = \textwidth]{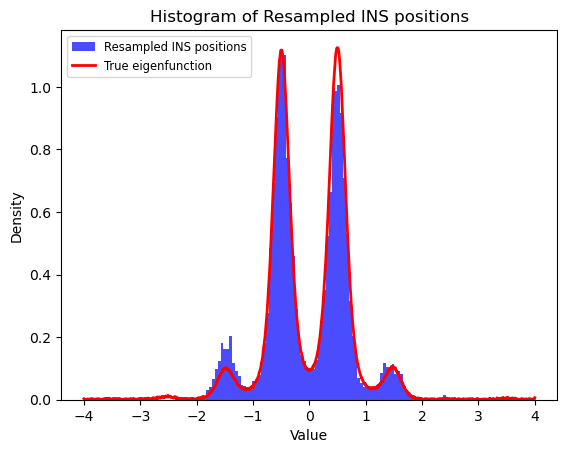}
    \caption{The INS system captures the true quasistationary distribution more accurately. }
    \end{subfigure}
    ~
    \begin{subfigure}{0.45\textwidth}
    \includegraphics[width = \textwidth]{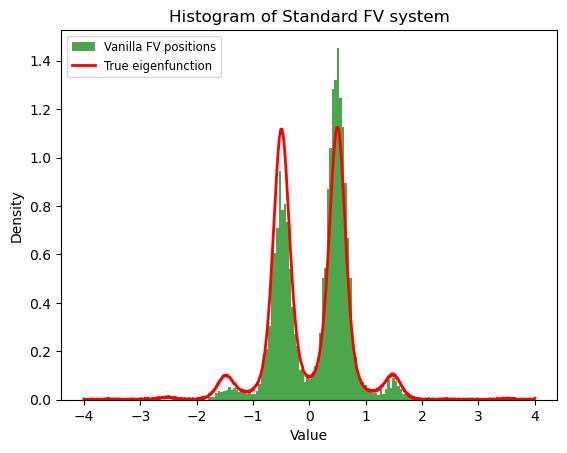}
    \caption{Multimodality is a challenge for an ensemble of independent Langevin trajectories. }
    \end{subfigure}
    \caption{ INS and standard Fleming-Viot particle systems approximation of the quasistationary distributions.  }
    \label{fig: qsdeig}
\end{figure}

\section{Concluding remarks}

In this final section we briefly mention issues in need of further attention.

\subsection{Allowing more than one killing/cloning event at a time}
The need to detect each killing/cloning event for the INS model,
or killing/cloning and swapping events for the non-INS model,
means that the diffusion dynamics cannot be simulated in parallel.
This slows the evolution of the simulated trajectories considerably.
An alternative for the INS model is to fix a deterministic time step,
simulate all diffusion dynamics out to that time,
identify all particles that would be killed/cloned in the interval,
and do an approximate resampling.
For the non-INS model one would also identify particles that would be swapped during the interval. 
It would seem that some modification of the algorithm along these lines will be needed for larger problems than those we have considered so far.

\subsection{Hard killing}
In the case of ``hard killing'' the forward process is killed as soon as it exits an open set $D$. 
This corresponds to a zero value Dirichlet boundary condition for the eigenfunctions,
with no zeroth order term for the one characterizing the QSD. 
As noted in Remark \ref{rem:jump}, 
the use of a Markov chain approximation for the diffusion dynamics has some advantages over more traditional approximations, 
both in terms of deciding when the process has exited $D$ and identifying the discretized boundary condition, 
which simply becomes an easily computed zeroth order term added to all grid points that can reach the complement of $D$ in one jump. 
The corresponding ergodic control problems attempt to confine the process to $D$ with least cost (sometimes called a ``state space constraint'' control problem).
Also,
one can consider combined soft and hard killing.

\subsection{Other swapping moves}
Once the reversibility structure needed for swapping between the forward and backward systems is identified,
one can consider any of the standard parameterizations used to identify other dynamics to use 
for accelerating convergence,
such as temperature.
The only caveat is that one must swap forward/backward pairs at different temperatures, 
rather than individual particles.

A second parameter one might consider that would play a role somewhat analogous to temperature is the discretization parameter $h$ used for the Markov chain approximation. 
If one thinks of the finest discretization as corresponding to the lowest temperature,
it is natural to couple with pairs that use coarser discretizations.
One expects the more coarsely discretized chain, since it uses larger (non-grid based) moves,
to move about the state space more quickly. 
In fact, what is perhaps most natural since the state spaces of the chains need not be related,
is to use both a coarser discretization and a higher temperature.


\subsection{Theoretical results}
Since the nonlinear Markov processes as well as the approximation by Fleming-Viot processes go beyond the theory developed in \cite{dupliupladol},
it would be of interest to show the optimality of the INS limit through an analogous large deviations analysis. 
In addition, 
obtaining a rate of convergence for the particle approximation 
would also be useful,
especially if it shed some light on the impact of how the number of particles $N$
interacts with the rate of convergence with respect to time.

\subsection{Implied potential for a general model}
The discussion of the implied potential in Example \ref{ex:periodicdiffusion_theory}
should be extended to the general case.
In particular, whether or not the implied potential is in fact always of a ``unimodal'' form before the approximation by a FV system seems an interesting question to resolve. 

\subsection{Generalizations}
Owing to the use of a nonlinear Markov process for the stochastic representation of the eigenvalue problem, it seems necessary to use a particle system to approximate the corresponding forward evolution.
Since one is already using such a particle approximation,
it makes sense to consider diffusion operators that also contain a mean-field dependence,
so long as there is the needed reversibility (this extension was suggested to us by Will Asness).

\appendix

\section{Proofs}

\begin{proof}
[Proof of Theorem \ref{lem:qsd-ef}]For the forward direction we can expand with
respect to $\Delta>0$ and find
\begin{align*}
&  \frac{1}{E_{\psi^{\varepsilon}}e^{-\int_{0}^{\Delta}c(X_{s}^{\varepsilon
})ds}}\cdot E_{\psi^{\varepsilon}}e^{-\int_{0}^{\Delta}c(X_{s}^{\varepsilon
})ds}f(X_{\Delta}^{\varepsilon})\\
&  \quad=\frac{1}{E_{\psi^{\varepsilon}}\left[  1-\Delta c(X_{0}^{\varepsilon
})+o(\Delta)\right]  }E_{\psi^{\varepsilon}}\left[  \left(  1-\Delta
c(X_{0}^{\varepsilon})+o(\Delta)\right)  f(X_{\Delta}^{\varepsilon})\right] \\
&  \quad=E_{\psi^{\varepsilon}}\left[  f(X_{\Delta}^{\varepsilon})\right]
-\Delta E_{\psi^{\varepsilon}}\left[  c(X_{0}^{\varepsilon})f(X_{0}%
^{\varepsilon})\right]  +\Delta E_{\psi^{\varepsilon}}\left[  c(X_{0}%
^{\varepsilon})\right]  E_{\psi^{\varepsilon}}\left[  f(X_{0}^{\varepsilon
})\right]  +o(\Delta).
\end{align*}
Using the definition of $\mathcal{L}^\varepsilon$ and also that of a QSD gives
\begin{align*}
o(\Delta)  &  =E_{\psi^{\varepsilon}}\left[  f(X_{\Delta}^{\varepsilon
})-f(X_{0}^{\varepsilon})\right]  -\Delta E_{\psi^{\varepsilon}}\left[
c(X_{0}^{\varepsilon})f(X_{0}^{\varepsilon})\right]  +\Delta E_{\psi
^{\varepsilon}}\left[  c(X_{0}^{\varepsilon})\right]  E_{\psi^{\varepsilon}%
}\left[  f(X_{0}^{\varepsilon})\right] \\
&  =\Delta\int_{\lbrack0,1)^{d}}\psi^{\varepsilon}(x)\mathcal{L}^{\varepsilon
}f(x)dx-\Delta\int_{\lbrack0,1)^{d}}\psi^{\varepsilon}(x)c(x)f(x)dx+\Delta
\int_{\lbrack0,1)^{d}}\psi^{\varepsilon}(x)c(x)dx\int_{[0,1)^{d}}%
\psi^{\varepsilon}(x)f(x)dx.
\end{align*}
Dividing by $\Delta$, sending $\Delta\rightarrow0$, and then using that $f$ is
arbitrary and $\lambda^{\varepsilon}=\int_{[0,1)^{d}}c(x)\psi^{\varepsilon
}(x)dx$ gives (\ref{eqn:QSD}).

The reverse is more complicated, and uses an an argument from the proof of
\cite[Part 6 of Theorem 2.3]{buddupnyqwu}. Let $\hat{p}_{t}^{\varepsilon
}(x,y)$ denote the Green's function for $\frac{\partial u}{\partial
t}=\mathcal{L}^{\varepsilon}u-cu$, so that
\[
E_{x}e^{-\int_{0}^{t}c(X_{s}^{\varepsilon})ds}f(X_{t}^{\varepsilon}%
)=\int_{[0,1)^{d}}\hat{p}_{t}^{\varepsilon}(x,y)f(y)dy.
\]
Note that $\hat{p}_{t}^{\varepsilon}(x,y)dy$ is a sub-probability. For a
smooth function $h(x,y)$, let $\mathcal{L}_{x}^{\varepsilon}h$ denote the
function obtained by regarding $\mathcal{L}^{\varepsilon}$ as an operator in
the variable $x$, and similarly for $\mathcal{L}_{y}^{\varepsilon,\ast}h$.
Then it follows \cite[page 155]{buddupnyqwu} that
\begin{equation}
\lbrack\mathcal{L}_{x}^{\varepsilon}-c(x)]\hat{p}_{t}^{\varepsilon
}(x,y)=[\mathcal{L}_{y}^{\varepsilon,\ast}-c(y)]\hat{p}_{t}^{\varepsilon
}(x,y). \label{eqn:rev}%
\end{equation}

Next append an absorbing state $\mathfrak{C}$ to $[0,1)^{d}$, and let $\hat
{X}_{t}^{\varepsilon}\in\mathfrak{C}$ for $t\geq\tau_{\mathfrak{C}}$, where
$\tau_{\mathfrak{C}}$ is the killing time. Given $f$ such that $\mathcal{L}%
^{\varepsilon}f$ is bounded and continuous on $[0,1]^{d}$, extend it to
$[0,1]^{d}\cup\mathfrak{C}$ by $f(\mathfrak{C})=0$, so that
\[
E_{x}f(\hat{X}_{t}^{\varepsilon})=\int_{[0,1)^{d}}\hat{p}_{t}^{\varepsilon
}(x,y)f(y)dy.
\]
Then with $\left\langle \psi^{\varepsilon},f\right\rangle \doteq\int
_{\lbrack0,1)^{d}}\psi^{\varepsilon}(x)f(x)dx$
\begin{align*}
E_{\psi^{\varepsilon}}f(\hat{X}_{t}^{\varepsilon})  &  =E_{\psi^{\varepsilon}%
}f(X_{t\wedge\tau_{\mathfrak{C}}}^{\varepsilon})=\left\langle \psi
^{\varepsilon},f\right\rangle +E_{\psi^{\varepsilon}}\left[  \int_{0}%
^{t\wedge\tau_{\mathfrak{C}}}\mathcal{L}^{\varepsilon}f(X_{s}^{\varepsilon
})ds\right] \\
&  =\left\langle \psi^{\varepsilon},f\right\rangle +E_{\psi^{\varepsilon}%
}\left[  \int_{0}^{t}\mathcal{L}^{\varepsilon}f(\hat{X}_{s}^{\varepsilon
})ds\right]
\end{align*}
and therefore
\begin{align*}
E_{\psi^{\varepsilon}}\left[  f(X_{t}^{\varepsilon})1_{\{\tau_{\mathfrak{C}%
}>t\}}\right]   &  =E_{\psi^{\varepsilon}}f(\hat{X}_{t}^{\varepsilon})\\
&  =\left\langle \psi^{\varepsilon},f\right\rangle +\int_{0}^{t}%
\int_{[0,1)^{d}}\int_{[0,1)^{d}}[\mathcal{L}^{\varepsilon}-c(y)]f(y)\hat
{p}_{s}^{\varepsilon}(x,y)dy\psi^{\varepsilon}(x)dxds\\
&  =\left\langle \psi^{\varepsilon},f\right\rangle +\int_{0}^{t}%
\int_{[0,1)^{d}}\int_{[0,1)^{d}}f(y)\left[  \mathcal{L}_{y}^{\varepsilon,\ast
}-c(y)\right]  \hat{p}_{s}^{\varepsilon}(x,y)dy\psi^{\varepsilon}(x)dxds.
\end{align*}
Then using (\ref{eqn:rev})%
\begin{align*}
E_{\psi^{\varepsilon}}\left[  f(X_{t}^{\varepsilon})1_{\{\tau_{\mathfrak{C}%
}>t\}}\right]   &  =\left\langle \psi^{\varepsilon},f\right\rangle +\int
_{0}^{t}\int_{[0,1)^{d}}\int_{[0,1)^{d}}f(y)[\mathcal{L}_{x}^{\varepsilon
,\ast}-c(x)]\hat{p}_{s}^{\varepsilon}(x,y)dy\psi^{\varepsilon}(x)dxds\\
&  =\left\langle \psi^{\varepsilon},f\right\rangle +\int_{0}^{t}%
\int_{[0,1)^{d}}\int_{[0,1)^{d}}f(y)\hat{p}_{s}^{\varepsilon}%
(x,y)dy[\mathcal{L}_{x}^{\varepsilon}-c(x)]\psi^{\varepsilon}(x)dxds\\
&  =\left\langle \psi^{\varepsilon},f\right\rangle -\lambda^{\varepsilon}%
\int_{0}^{t}\int_{[0,1)^{d}}\int_{[0,1)^{d}}f(y)\hat{p}_{s}^{\varepsilon
}(x,y)dy\psi^{\varepsilon}(x)dxds\\
&  =\left\langle \psi^{\varepsilon},f\right\rangle -\lambda^{\varepsilon}%
\int_{0}^{t}E_{\psi^{\varepsilon}}\left[  f(X_{s}^{\varepsilon})1_{\{\tau
_{\mathfrak{C}}>s\}}\right]ds  .
\end{align*}
Hence
\[
E_{\psi^{\varepsilon}}\left[  f(X_{t}^{\varepsilon})1_{\{\tau_{\mathfrak{C}%
}>t\}}\right]  =e^{-\lambda^{\varepsilon}t}\left\langle \psi^{\varepsilon
},f\right\rangle =e^{-\lambda^{\varepsilon}t}\int_{[0,1)^{d}}\psi
^{\varepsilon}(x)f(x)dx,
\]
and $E_{\psi^{\varepsilon}}\left[  1_{\{\tau_{\mathfrak{C}}>t\}}\right]
=e^{-\lambda^{\varepsilon}t}$. We therefore obtain
\[
\int_{\lbrack0,1)^{d}}\psi^{\varepsilon}(x)f(x)dx=\frac{1}{E_{\psi
^{\varepsilon}}\left[  1_{\{\tau_{\mathfrak{C}}>t\}}\right]  }\cdot
E_{\psi^{\varepsilon}}\left[  f(X_{t}^{\varepsilon})1_{\{\tau_{\mathfrak{C}%
}>t\}}\right]
\]
for any $t\in(0,\infty)$, and so $\psi^{\varepsilon}(x)dx$ is a QSD.
\end{proof}

\section{Pseudocode}
\label{sec:app:pseudocode}

\begin{algorithm}
\caption{Compute SDE Transition Rates}
\begin{algorithmic}[1]
\Function{Rates}{$x, b(x), a(x), h$}
\State \textbf{Initialize:} $r_+, r_- \in \R^d$
\For{$i = 1$ to $d$}
    \State $r_{+, i} \gets \frac{1}{2h^2}(h  b_i(x) + a_{ii}(x))$
    \State $r_{+, i} \gets \frac{1}{2h^2}(-h  b_i(x) + a_{ii}(x))$
\EndFor

\If{$\exists i$ such that $r_{+, i} < 0$ or $r_{-, i} < 0$}
    \For{$i = 1$ to $d$}
        \State $r_{+, i} \gets \frac{1}{h^2}(h  \max(b_i(x), 0) + a_{ii}(x)/2)$
        \State $r_{-, i} \gets \frac{1}{h^2}(h  \max(-b_i(x), 0) + a_{ii}(x)/2)$
    \EndFor
\EndIf

\State $r \gets \text{concatenate}(r_+, r_-)$
\State \Return $r$
\EndFunction
\end{algorithmic}
\label{alg:sdetransrate}
\end{algorithm}

\begin{algorithm}
\caption{One Step Transition of SDE Jump Process}
\begin{algorithmic}[1]
\Function{OneStep}{$x,r,h$}

\State $r_{\text{total}} \gets \sum_{i=1}^{2d} r_i$
\State $p \gets r / r_{\text{total}}$ \Comment{Normalize rates to obtain probabilities}
\State $j \gets \text{random sample from } \{1, \dots, 2d\} \text{ with probabilities } p$

\If{$j \leq d$}
    \State $x_{\text{new}, j} \gets x_j + h$
\Else
    \State $x_{\text{new}, j - d} \gets x_{j-d} - h$
\EndIf

\State \Return $x_{\text{new}}$
\EndFunction
\end{algorithmic}
\label{alg:sdetransition}
\end{algorithm}

\begin{algorithm}
\caption{Stable computation of infinite swap rate}
\begin{algorithmic}[1]
\Function{InfSwap}{$x,y,V,\varepsilon$}
\State $\rho_\varepsilon(x,y) \gets \left(\exp\left(\frac{2(V(x) - V(y))}{\varepsilon}\right) + 1\right)^{-1}$
\State \Return $\rho_\varepsilon(x,y)$
\EndFunction
\end{algorithmic}
\label{alg:infswapcomputation}
\end{algorithm}

\begin{algorithm}
\caption{Compute Event Rates for Particle $x$}
\begin{algorithmic}[1]
\Function{EventRates}{$x,\rho_\varepsilon, DV, \Delta V, a, c,h$}

\State $b_{\text{symm}}(x) \gets (1 - 2 \rho_\varepsilon)DV(x)$
\State $r_{dyn} \gets  \textsc{Rates}(x, b_{\text{symm}}(x), a(x), h)$ \Comment{Dynamics rates}

\State $r_{kc} \gets c(x) - (1- \rho_\varepsilon)\Delta V(x)$ \Comment{Kill/Clone rate}
\State $r_{{net}} \gets \sum_{i = 1}^{2d} r_{{dyn}_i} + |r_{kc}|$ \Comment{Net event rate}

\State \Return $r_{\text{net}}, r_{\text{dynamics}}, r_{\text{kill}}$
\EndFunction
\end{algorithmic}
\label{alg:eventrates}
\end{algorithm}

\begin{algorithm}
\caption{Killing-Cloning Event for particle $x_i$}

\begin{algorithmic}[1]
\Function{KillClone}{ $\{x^j \}_{j = 1}^N$, $\{y^j\}_{j = 1}^N$, $\{ \rho_\varepsilon^j\}_{j = 1}^N$, $i$, $c^i_{symm}$ } 

\State $i' \gets$ Random Choice from $\{1, \ldots, N\}$
\If{$i' = i$}
    \State $x^{i} \gets x^{i'}$ \Comment{Do not move particle}
    \State $\mathcal{I} \gets [i']$ \Comment{Index to recompute clock}
\Else
    \If{$c_{symm}^i > 0$} \Comment{Kill and respawn}
        \If{Uniform$(0,1)$$< \rho_\varepsilon^i$} \Comment{Determine forward or backward}
            \If{Uniform$(0,1)$$< \rho_\varepsilon^{i'}$}
                \State $x^i \gets x^{i'}$
            \Else
                \State $x^i \gets y^{i'}$
            \EndIf
        \Else
            \If{Uniform$(0,1)$ $< 1 - \rho_\varepsilon^{i'}$}
                \State $x^i \gets y^{i'}$
            \Else
                \State $x^i \gets x^{i'}$
            \EndIf
        \EndIf
        \State $\mathcal{I}\gets [i']$
    \Else \Comment{Clone and cull}
        \If{Uniform$(0,1)$$< \rho_\varepsilon^i$}
            \If{Uniform$(0,1)$$< \rho_\varepsilon^{i'}$}
                \State $x^{i'}\gets x^i$
            \Else
                \State $y^{i'}\gets x^i$
            \EndIf
        \Else
            \If{Uniform$(0,1)$ $< 1 - \rho_\varepsilon^{i'}$}
                \State $y^{i'}\gets x^i$
            \Else
                \State $x^{i'} \gets x^i$
            \EndIf
        \EndIf
        \State $\mathcal{I} \gets [i, i']$
    \EndIf
\EndIf
\State \Return $\{x^j\}_{j = 1}^N, \{y^j\}_{j = 1}^N, \mathcal{I}$
\EndFunction
\end{algorithmic}
\label{alg:killclone}
\end{algorithm}

\begin{algorithm}
\caption{Backward-Forward Infinite Swapping Fleming-Viot}
\begin{algorithmic}[1]
\Function{InfSwapFlemingViot}{ $T$, $V$, $DV$, $\Delta V$, $a$, $\varepsilon$, $c$, $h$, $\{x_0^{(j)}\}_{j = 1}^N$, $\{y_0^{(j)}\}_{j = 1}^N$ }

\State $t_{now} \gets 0$, $t_{\text{all}} = \{t_{\text{now} }\}$,  $m \gets 0$
\State $\mathcal{I} = \{1,\ldots, N \}$


\While{$t_{\text{now}} < T$}
    \For{$i$ in $\mathcal{I}$}
    \State $\rho_\varepsilon^i \gets$ \textsc{InfSwap}($x_m^{(i)}$, $y_m^{(i)}$, $V$, $\varepsilon$)
    \State $r^{x,i}_{\text{net}}, r^{x,i}_{\text{dyn}}, r^{x,i}_{\text{kill}}\gets$ \textsc{EventRates}($x^{(i)}_m$, $\rho_\varepsilon^{(i)}$, $DV$, $\Delta V$, $a$, $c$, $h$)
    \State $r^{y,i}_{\text{net}}, r^{y,i}_{\text{dyn}}, r^{y,i}_{\text{kill}}\gets$ \textsc{EventRates}($y^{(i)}_m$, $1-\rho_\varepsilon^{(i)}$, $DV$, $\Delta V$, $a$, $c$, $h$)
    \State $\tau^{x,i} \gets$ Exponential$(1/r^{x,i}_{\text{net}})$
    \State $\tau^{y,i} \gets$ Exponential$(1/r^{y,i}_{\text{net}})$
\EndFor

    \State $i' \gets \arg\min(\{\tau^{x,j}\}_{j = 1}^N, \{\tau^{y,j}\}_{j = 1}^N)$,  $t_\Delta \gets \min(\{\tau^{x,j}\}_{j = 1}^N, \{\tau^{y,j}\}_{j = 1}^N)$
    \State $t_{\text{now}} \gets t_{\text{now}} + t_\Delta$
    \State $t_{\text{all}} \gets t_{\text{all}} \cup t_{\text{now}} $
    
    \Comment{Update residual times}
    \For{$i = 1$ to $N$}
        \State $\tau^{x,i} \gets \tau^{x,i} - t_\Delta$, $\tau^{y,i} \gets \tau^{y,i} - t_\Delta$
    \EndFor

    \Comment{Determine if x or y event}
    \If{$i' < N$}
        \State $p_1 \gets \{x_{m}^{(j)}\}_{j = 1}^N$, $p_2 \gets \{y_{m}^{(j)}\}_{j = 1}^N$
        \State $r_{\text{dyn}} \gets r_{\text{dyn}}^{x,i'}$,\, $r_{\text{kill}} \gets r_{\text{kill}}^{x,i'}$,\, $r_{\text{net}} \gets r^{x,i'}_{\text{net}}$
        \State $\rho_1 \gets \{\rho_\varepsilon^{j}\}_{j = 1}^N$
    \Else
        \State $p_1 \gets \{y_{m}^{(j)}\}_{j = 1}^N$, $p_2 \gets \{x_{m}^{(j)}\}_{j = 1}^N$
        \State $r_{\text{dyn}} \gets r_{\text{dyn}}^{y,i'}$,\, $r_{\text{kill}} \gets r_{\text{kill}}^{y,i'}$,\, $r_{\text{net}} \gets r^{y,i'}_{\text{net}}$
        \State $\rho_1 \gets \{1-\rho_\varepsilon^{j}\}_{j = 1}^N$
    \EndIf

    \State $i \gets i' \mod N$

    \Comment{Determine event type}
    \If{Uniform$(0,1)$ $< \sum r_{\text{dyn}}/r_{\text{net}}$}
        \State $p_1^{i} \gets$ \textsc{OneStep}($p_1^{i}$, $r_{\text{dyn}}$, $h$)
        \State $\mathcal{I} \gets [i]$
    \Else
        \State $p_1, p_2, \mathcal{I} \gets$ \textsc{KillClone}($p_1$, $p_2$,  $\rho_1$, $i$,$r_{\text{kill}}^{i}$)
    \EndIf
    \State $m \gets m+1$

    \If{$i' < N$}
        \State $\{ x_m^{(j)} \}_{j = 1}^N \gets p_1$, $\{y_m^{(j)} \}_{j = 1}^N \gets p_2$
    \Else
        \State $\{y_m^{(j)} \}_{j = 1}^N\gets p_1$,  $\{ x_m^{(j)} \}_{j = 1}^N\gets p_1$
    \EndIf

\EndWhile

\State \Return $\{x_m^j \}_{j = 1}^N, \{y_m^j \}_{j = 1}^N, t_{\text{all}}$

\EndFunction

\end{algorithmic}
\label{alg:FVINS}
\end{algorithm}

\newpage

\bibliographystyle{plain}
\bibliography{main}

\end{document}